\documentclass[11pt]{amsart}
\usepackage[utf8]{inputenc}
\usepackage{mathtools}
    \mathtoolsset{centercolon}
\usepackage{amssymb}
\usepackage{enumitem}
\usepackage[hyphens]{url} 
\usepackage[hyphenbreaks]{breakurl}
\usepackage{hyperref}
\hypersetup{colorlinks=true}
\usepackage{scalerel}
\usepackage{thmtools, thm-restate}
\usepackage{tikz-cd}
\usepackage{tikz}

\definecolor{mypink}{RGB}{255, 0, 128}
\definecolor{mypurple}{RGB}{64, 0, 128}

\theoremstyle{plain}
  \declaretheorem[within=section]{theorem}

  \declaretheorem[name={Theorem}]{maintheorem}

  \declaretheorem[numbered=no,name={Theorem}]{theorem*}
  \declaretheorem[sibling=theorem]{proposition}
  \declaretheorem[sibling=theorem]{lemma}
  \declaretheorem{claim}
  \declaretheorem[sibling=theorem]{corollary}
\theoremstyle{definition}
  \declaretheorem[sibling=theorem]{remark}
  \declaretheorem[sibling=theorem]{definition}
  
  \declaretheorem[sibling=theorem]{question}
  
  \declaretheorem[sibling=theorem,qed={$\diamondsuit$}]{example}

\newcommand{\idyll}{\mathrm{idyll}}
\newcommand{\oblpr}{\mathrm{oblpr}}

\newcommand{\vf}{\varphi}

\newcommand{\B}{\mathbf{B}}
\newcommand{\C}{\mathbf{C}}
\newcommand{\F}{\mathbf{F}}
\newcommand{\G}{\mathbf{G}}
\newcommand{\K}{\mathbf{K}}
\newcommand{\N}{\mathbf{N}}
\renewcommand{\P}{\mathbf{P}}
\newcommand{\Q}{\mathbf{Q}}
\newcommand{\R}{\mathbf{R}}
\renewcommand{\S}{\mathbf{S}}
\newcommand{\T}{\mathbf{T}}

\newcommand{\OC}{\mathcal{O}_{C}}

\newcommand{\TR}{ 
    {\mathbf{T}\!\!\mathbf{R}}
}
\newcommand{\TP}{ 
    {\mathbf{T}\!\!\mathbf{P}}
}
\newcommand{\TC}{ 
    {\mathbf{T}\!\mathbf{C}}
}
\newcommand{\Z}{\mathbf{Z}}

\newcommand{\lex}{\le_{\rm lex}}

\newcommand{\claimqed}{\hfill $\diamondsuit$\vspace{8pt}}

\DeclareMathOperator{\eq}{eq}
\DeclareMathOperator{\ev}{ev}
\DeclareMathOperator{\Ext}{Ext}

\DeclareMathOperator{\im}{im}
\DeclareMathOperator{\In}{in}
\DeclareMathOperator{\lc}{lc}
\DeclareMathOperator{\LCH}{LCH}

\DeclareMathOperator{\mult}{mult}
\DeclareMathOperator{\newt}{Newt}
\DeclareMathOperator{\sign}{sign}

\DeclareMathOperator*{\conv}{conv}
\DeclareMathOperator*{\bigboxplus}{\raisebox{-0.5em}{\scaleobj{2}{\boxplus}}}

\usepackage[
    backend=biber,
    style=alphabetic,
    url=false,
    doi=true,
    eprint=true
]{biblatex}
\title[Tropical Extensions and B-L Multiplicities for Idylls]{Tropical Extensions and Baker-Lorscheid Multiplicities for Idylls}
\author{Trevor Gunn}
\email{\href{mailto:tgunn@gatech.edu}{tgunn@gatech.edu}}
\address{School of Mathematics, Georgia Institute of Technology, Atlanta, USA}

\begin{document}
\begin{abstract}
In a recent paper, Matthew Baker and Oliver Lorscheid showed that Descartes's Rule of Signs and Newton's Polygon Rule can both be interpreted as multiplicities of polynomials over hyperfields. Hyperfields are a generalization of fields which encode things like the arithmetic of signs or of absolute values. By looking at multiplicities of polynomials over such algebras, Baker and Lorscheid showed that you can recover the rules of Descartes and Newton.

In this paper, we define tropical extensions for idylls. Such extensions have appeared for semirings with negation symmetries in the work of Akian-Gaubert-Guterman and for hypergroups and hyperfields in the work of Bowler-Su. Examples of tropical extensions are extending the tropical hyperfield to higher ranks, or extending the hyperfield of signs to the tropical real hyperfield by including a valuation.

The results of this paper concern the interaction of multiplicities and tropical extensions. First, there is a lifting theorem from initial forms to the entire polynomial from which we will show that multiplicities for a polynomial are equal to the corresponding multiplicity for some initial form. Second, we show that tropical extensions preserve the property that the sum of all multiplicities is bounded by the degree. Consequentially, we have this degree bound for every stringent hyperfield. This gives a partial answer to a question posed by Baker and Lorscheid about which hyperfields have this property.

\bigskip


\bigskip

\end{abstract}
\maketitle

\section{Introduction}
Let $P$ be a polynomial over a field $K$, where $K$ has some additional structure like an absolute value or a total order. Two classical problems are determining the relationship between the absolute values or the signs of the coefficients and those of the roots. Newton's rule describes the relationship between a non-Archimedean valuation of the coefficients and of the roots. Descartes's rule describes the number of positive roots or negative roots with respect to the pattern of signs of the coefficients.

Recently, Matthew Baker and Oliver Lorscheid \cite{BL} put these kinds of questions into a common framework known as hyperfields, which are algebras%
\footnote{The term ``algebra'' is used in this paper in the broad sense of a set with some distinguished elements, operations and relations.}
which capture the arithmetic of signs or of absolute values. For instance, $\S \coloneqq \R / \R_{> 0} = \{[0], [1], [-1]\}$ is the hyperfield of signs. Multiplication and addition in $\S$ come from the quotient. I.e.\ $[a][b] = [ab]$ and addition of equivalence classes is given by $\sum [a_i] = \{[\sum a_i'] : a_i' \in [a_i]\}$. For example, $[1] + [1] = \{[1]\}$ and $[1] + [-1] = \{[0], [1], [-1]\}$.

Let us look at some example questions which Baker and Lorscheid's framework addresses.

\begin{example}
    Consider the polynomial
    \[
        F(x) = (x + 3)(x - 4)(x - 6) = 2^3 \cdot 3^2 - 2\cdot 3x - 7x^2 + x^3.
    \]
    The sign sequence of the coefficients is $+, -, -, +$ and the signs of the roots are $-, +, +$. Descartes's rule of signs says that after removing any zeroes from the coefficient sequence, the number of positive roots we should expect is equal to the number of adjacent pairs of opposite signs in the coefficient sequence.
    
    For the number of negative roots, we look at $F(-x)$. So if there are no zero coefficients, then the number of negative roots we expect is equal to the number of adjacent pairs of identical signs in the coefficient sequence. Moreover, this bound is sharp so long as all the roots are real.

    Baker and Lorscheid consider this question over the hyperfield of signs. Specifically, if we take the polynomial $f = [1] + [-1]x + [-1]x^2 + [1]x^3$ over the hyperfield of signs, then their multiplicity operator (\ref{def:mult}) gives $\mult^\S_{[1]}(f) = 2$ and $\mult^\S_{[-1]}(f) = 1$.
\end{example}

\begin{example} \label{ex:first-trop}
    Next, consider the same polynomial but with the $2$-adic or $3$-adic valuation. Here we make a scatter plot of $(c, v_p(c))$ for each coefficient $c$ and $p \in \{2, 3\}$ and then take the lower convex hull as shown in in Figure~\ref{fig:intro}.

    \begin{figure}[htbp]
        \label{fig:intro}
        \centering
        \begin{tikzpicture}
            \foreach \x/\y in {0/3,1/1,2/0,3/0}
                \filldraw (\x,\y) circle (0.25mm);
            \draw (0,3) -- (1,1) -- (2,0) -- (3,0);

            \draw (0.5,2) node [left] {slope $-2$};
            \draw (1.5,0.5) node [left] {slope $-1$};
        \end{tikzpicture}
        \hspace{1cm}
        \begin{tikzpicture}
            \foreach \x/\y in {0/2,1/1,2/0,3/0}
                \filldraw (\x,\y) circle (0.25mm);
            \draw (0,2) -- (1,1) -- (2,0) -- (3,0);

            \draw (1,1) node [left] {slope $-1$};
        \end{tikzpicture}
        \caption{Newton polygons of $(x + 3)(x - 4)(x - 6)$ in $\Q_2$ and $\Q_3$ respectively}
    \end{figure}
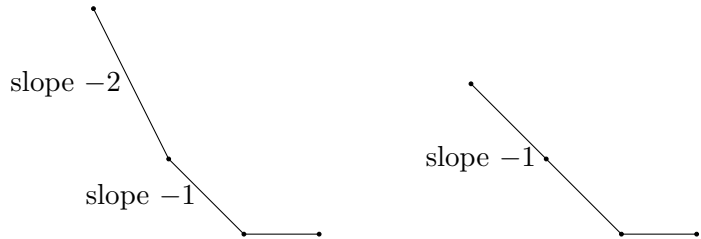

    Newton's Polygon Rule says that the number of roots $r$ with $v_p(r) = k$ is equal to the horizontal width the edge with slope $-k$ (i.e. its length after projecting to the $x$-axis). Thus, for $p = 2$, the valuations of the roots are $0,1,2$ and for $p = 3$ they are $0, 1, 1$.

    Likewise, if we consider the polynomials $3 + 1x + 0x^2 + 0x^3$ and $2 + 1x + 0x^2 + 0x^3$ over the tropical hyperfield ($\T$), Baker and Lorscheid's multiplicity operator, $\mult^\T$, gives the numbers above. For example, $\mult^\T_1 (2 + 1x + 0x^2 + 0x^3) = 2$.
\end{example}

In this paper, we will look at how their multiplicity operator works in the context of a tropical extension. The most common and natural examples of tropical extension are as follows: take $K/G$ to be a hyperfield coming from a quotient, and form a field of series over $K$ (e.g.\ Laurent or Puiseux series). Then quotient by the group of series whose leading coefficient belongs to $G$. For example, if the hyperfield is $\R/\R_{>0}$, then the equivalence classes in this tropical extension are $[0]$ and $\{[\pm t^n] : n \in \Gamma\}$ where the ordered group $\Gamma$ depends on what sort of series we use. Arithmetic in this hyperfield is a combination of the arithmetic of signs and of non-Archimedean absolute values and is described in detail in \cite{G}.

We also address so-called \emph{stringent} hyperfields---a term introduced by Nathan Bowler and Ting Su \cite{BS}. A hyperfield is stringent if $a \boxplus b$ is a singleton if $a \ne -b$. Stringent hyperfields are the next simplest form of hyperfields after of fields. We show that for a polynomial over a stringent hyperfield, the sum of all the multiplicities is bounded by the polynomial's degree (Corollary~\ref{cor:stringent}).

\subsection{Structure of the paper and a rough statement of the results}
In this paper, the primary type of algebra are \emph{idylls}---a generalization of fields which consists of a monoid $B^\bullet$ describing multiplication and a proper ideal $N_{B} \subseteq \N[B^\bullet]$ describing addition.
To describe these algebras, it will be convenient to talk about the larger category of ordered blueprints introduced by Lorscheid \cite{L, LF1, LGB1, LGB2, LSTT} which describe addition through a preorder on $\N[B^\bullet]$.
An Euler diagram of the relationships between these categories is show in Figure~\ref{fig:euler-diagram}.

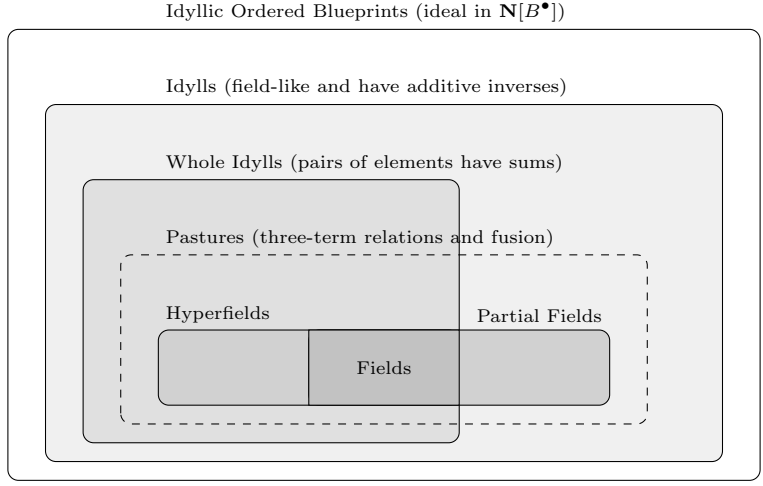
\begin{figure}[htbp]
    \begin{tikzpicture} \tiny
        \draw[rounded corners] (0, 0) rectangle (10, 6);
        \draw (2, 6) node[above right] {Idyllic Ordered Blueprints (ideal in $\N[B^\bullet]$)};
        \draw[rounded corners, fill=black!6] (0.5, 0.25) rectangle (9.5, 5);
        \draw (2, 5) node[above right] {Idylls (field-like and have additive inverses)};
        \draw[rounded corners, fill=black!12] (1, 0.5) rectangle (6, 4);
        \draw (2, 4) node[above right] {Whole Idylls (pairs of elements have sums)};
        \draw[rounded corners, dashed] (1.5, 0.75) rectangle (8.5, 3);
        \draw (2, 3) node[above right] {Pastures (three-term relations and fusion)};
        \draw[rounded corners, fill=black!18] (6, 1) rectangle (2, 2) node[above right] {Hyperfields};
        \draw[rounded corners, fill=black!18] (4, 1) rectangle (8, 2) node[above left] {Partial Fields};
        \draw[fill=black!24] (4, 1) rectangle (6, 2);
        \draw (5, 1.5) node {Fields};
    \end{tikzpicture}
    \caption{Euler diagram of relationships between sub-categories of ordered blueprints}
    \label{fig:euler-diagram}
\end{figure}

We will first state some rough definitions and results here, giving as many definitions as we reasonably can. A more thorough description of ordered blueprints and idylls is given in Section~2. In Section~3, we define polynomial extensions and tropical extensions and discuss Newton polygons and initial forms. In Section~4, we finish describing the theory of polynomials and multiplicities, factorization and multiplicities. In Section~5, we show that tropical extensions for hyperfields (after Bowler and Su \cite{BS}) are a special case of tropical extensions of idylls. In Section~6, we prove Theorems~A, B, C which concern initial forms and lifting. In Section~7, we give some examples and corollaries connecting this work to previous results and prove Theorem~D and Corollary~E concerning the degree bound. In the appendix, we record some division algorithms which have appeared in \cite{BL}, \cite{G} and \cite{AL}.

\begin{definition}
    An ordered blueprint consists of two parts: a multiplicative and additive structure. Multiplication is defined by a monoid $(B^\bullet, 0_B, 1_B, \cdot_B)$ with identity $1_B$ and an absorbing element $0_B$. The additive structure is defined by an additive and multiplicative preorder among formal sums over $B^\bullet$~\cite{LSTT}.
\end{definition}

Within the category of ordered blueprints, are what are named here \emph{idyllic} ordered blueprints for which this preorder is entirely described by an ideal $N_B \coloneqq \{ \sum a_i \in \N[B^\bullet] : 0 \leqslant \sum a_i \}$. The category of idyllic ordered blueprints is morally speaking, the smallest category containing hyperfields, partial fields, and polynomial extensions. We work entirely within this category and often within the sub-subcategory of field-like objects which we call idylls.

\begin{restatable}{definition}{idylldef} \label{def:idyll}
    An \emph{idyll} $B$, is a pair $(B^\bullet, N_B)$, consisting of a monoid $B^\bullet = (B^\bullet, 0_B, 1_B, \cdot_B)$ together with a proper ideal $N_B$ of $\N[B^\bullet]$, and which is ``field-like'' in the sense that:
    \begin{itemize}
        \item $0_B \neq 1_B$,
        \item $B^\times \coloneqq B^\bullet \setminus \{0_B\}$ is a group,
        \item there exists a unique $\epsilon_B \in B^\bullet$ such that $\epsilon_B^2 = 1$ and $1 + \epsilon_B \in N_B$.
    \end{itemize}    
    $N_B$ is called the \emph{null-ideal} of $B$ and $B^\bullet$ and $B^\times$ are called the \emph{underlying monoid} and \emph{group of units} respectively.
\end{restatable}

A first example of an idyll is the idyll associated to a field.

\begin{example}
    Let $K$ be a field and let $K^\bullet = (K, 0_K, 1_K, \cdot_K)$ be the multiplicative monoid of $K$. Then we can define $N_K$ as the ideal of all formal sums whose image in $K$ is $0$.
\end{example}

Next, we have the idylls associated to the rules of Descartes and Newton.
More examples of idylls will be given in Section~\ref{sec:idylls}.

\begin{example}
    The \emph{idyll of signs} or \emph{sign idyll}, $\S$, has underlying monoid $\S^\bullet = \{0, 1, -1\}$ with the standard multiplication. The null-ideal of $\S$ is the set of all formal sums that include at least one $1$ and at least one $-1$. In other words, a formal sum of signs $\sum s_i$ is in $N_\S$ if and only if there exists real numbers $x_i$ such that $\sign(x_i) = s_i$ and $\sum x_i = 0$ in $\R$.
\end{example}

\begin{example}
    The \emph{tropical idyll}, $\T$, is the idyll whose underlying monoid is $(\R \cup \{\infty\}, \infty, 0, +)$, where $\infty$ is an absorbing element for the monoid. The null-ideal, $N_\T$, is the set of all formal sums where the minimum term (in the usual ordering) appears at least twice in the sum. In other words, a formal sum of valuations $\sum \gamma_i$ is in $N_\T$ if and only if there is a valued field $(K, v)$ containing elements $x_i$ such that $v(x_i) = \gamma_i$ and $\sum x_i = 0$ in $K$.
\end{example}

We now introduce the concept of a \emph{tropical extension}. The classical analogue of this is to take a field $K$ and form the field of Laurent series or Puiseux series in $t$ over $K$. This gives us a $t$-adic valuation where the residue field is $K$. Likewise, we are here forming a larger idyll with a valuation and whose ``residue idyll'' is the idyll we start with.
We leave some categorical constructions to Section~\ref{sec:cat-constructions}.

\begin{restatable}{definition}{tropextdef} \label{def:tropext}
    If $B$ is an idyll with multiplicative group $B^\times$, then a \emph{tropical extension} of an ordered Abelian group $\Gamma$ by $B$ is an idyll $C$ with some additional properties. First, there are morphisms $B \xrightarrow{\iota} C \xrightarrow{v} \Gamma$ which induce a short exact sequence of groups:
    \[
        1 \to B^\times \xrightarrow{\iota^\bullet} C^\times \xrightarrow{v^\bullet} \Gamma \to 1.
    \]
    Second, the exactness of the sequence of groups must extend to the ordered blueprints, i.e.\ $\im(\iota) = \eq(v, 1)$.
    Lastly, we require that $N_C$ has the property that $\sum c_i \in N_C$ if and only if $\sum_I c_i \in N_C$ where $I = \{i : v^\bullet(c_i) \text{ is minimal}\}$.

    With a slight abuse of notation, we will write $C \in \Ext^1(\Gamma, B)$ to mean that $C$ is a tropical extension of $\Gamma$ by $B$.
\end{restatable}

\begin{remark}
    Tropical extensions appear in the work of Akian, Gaubert, and Guterman for semirings with a symmetry (negation)~\cite{AGG} as well as in the work of Rowan for the more general setting of ``semiring systems'' \cite{R}.

    For (skew) hyperfields, tropical extensions appear in the work of Bowler and Su as a semidirect product~\cite{BS}. Some examples of this are as follows.
\end{remark}

\begin{example} \label{ex:tropext-firstexamples}
    The most basic example of a tropical extension is the tropical idyll itself, which fits into an exact sequence
    \[
        1 \to \K^\times \to \T^\times \xrightarrow{\sim} \R \to 1. \qedhere
    \]
\end{example}

\begin{example}
    More generally, let $\T_m = (\R^m, \lex)^\idyll$ be the rank-$m$ tropical idyll, which is defined the same way as $\T$ but using $(\R^m, \lex)$ in place of $(\R, \le)$. For all $m, n$, we have a tropical extension
    \[
        1 \to \T_m^\times \to \T_{m + n}^\times \to \R^n \to 1. \qedhere
    \]
\end{example}

\begin{example}
    The tropical real idyll, $\TR$, is the extension
    \[
        1 \to \S^\times \to \TR^\times \to \R \to 1.
    \]
    Here $\TR^\bullet = \{\pm t^\gamma : \gamma \in \R\} \cup \{0\}$ with the natural multiplication. The null-ideal, $N_{\TR}$, is the set of all formal sums $\sum s_i t^{\gamma_i}$ such that if $I = \{i : \gamma_i \text{ is minimal}\}$ then $\sum_I s_i \in N_{\S}$. I.e.\ among the coefficients $\{s_i : i \in I\}$, there is at least one $+1$ and at least one $-1$.
    The tropical real idyll is described further in \cite{G}.
\end{example}

Similar to tropical extensions, we can define polynomial extensions over idylls and define a recursive multiplicity operator for roots of these polynomials.

We have the following definition of a multiplicity operator for idylls, which appears as Definition 1.5 of \cite{BL} for polynomials over hyperfields.
\begin{restatable}{definition}{defmult} \label{def:mult}
    Let $B$ be an idyll, let $f \in B[x]$ be a polynomial and let $a \in B^\bullet$. The \emph{multiplicity of $f$ at $a$} is
    \[
        \mult^B_a(f) = 1 + \max \mult^B_a(g)
    \]
    where the maximum is taken over all factorizations of $f$ into $(x - a)g$, or $\mult^B_a(f) = 0$ if there are no such factorizations.
\end{restatable}

\begin{definition}
    We say that $f$ factors into $(x - a)g$ if $f - (x - a)g$ belongs to the null ideal of $B[x]$. This is equivalent to saying that the degree $d$ terms of $f - (x - a)g$ belong to $x^d N_B$ for all $d$.
\end{definition}

We will define a generalization of leading coefficients and initial forms for tropical extensions. Specifically, we generalize two operators from the classical setting.
First, if $c = \sum a_i t^i \in \C((t))^\times$ is a nonzero Laurent series, then $\lc^\bullet(c) = a_{i_0} \in \C$ where $i_0 = \min\{i : a_i \neq 0\}$. Second, if $f = \sum c_i x^i \in \C((t))[x]$ and $w \in \Z$, then $\In_w(f) = \sum_I \lc^\bullet(c_i) x^i \in \C[x]$ where $I = \{i : v_t(c_i) + iw \text{ is minimal}\}$ and $v_t : \C((t)) \to \Z \cup \{\infty\}$ is the $t$-adic valuation.

For the main theorems, we need one more axiom. We define a \emph{whole} idyll to be an idyll for which every pair of elements $a, b \in B^\bullet$ has at least one `sum' $c \in B^\bullet$ for which $a + b - c \in N_B$. The class of whole idylls includes fields and hyperfields but excludes partial-fields which are not themselves fields. Additionally, let us be clear that a ``polynomial'' in this paper is not allowed to have multiple terms with the same degree, so $x + x^2 + x^5$ is a polynomial but $x + x + x$ is not.

With this in mind, the main theorem for split extensions (having a splitting $\Gamma \to C^\times, \gamma \mapsto t^\gamma$) is the following.
\begin{maintheorem} \label{thm:split}
    Let $B$ be a whole idyll and let $C = B[\Gamma]$ be a split tropical extension of $\Gamma$ by $B$. Then for every polynomial $f \in C[x]$ and $a \in C^\bullet$ with valuation $\gamma$,
    \[
        \mult^C_a(f) = \mult^B_{\lc^\bullet(a)}(\In_{\gamma}(f)).
    \]
\end{maintheorem}

With some slight modification to the ideas of initial forms, we extend this result to the non-split case as follows.
\begin{restatable}{maintheorem}{mainthm} \label{thm:main}
    Let $B$ be a whole idyll and let $C \in \Ext^1(\Gamma, B)$ be a tropical extension of $\Gamma$ by $B$. Let $f \in C[x]$ be a polynomial and let $a \in C^\bullet$ be a root of $f$. Then
    \[
        \mult^C_a(f) = \mult^{B}_{\lc^\bullet(a)} (\In_a(f)).
    \]
\end{restatable}

In proving this theorem, we will show the following result. Notation is the same as in the previous theorem.
\begin{restatable}{maintheorem}{liftingthm} \label{thm:lifting}
    Any factorization of $\In_a f$ into $(x - 1)g$ can be lifted to a factorization of $f$ into $(x - a)\tilde{g}$ such that $\In_a \tilde{g} = g$.
\end{restatable}

\begin{example}
    If $f$ is a tropical polynomial, then its multiplicity at $w \in \T^\times$ is the same as the multiplicity of the initial form $\In_w f$ and we will see in Example~\ref{ex:krasner-factorization} that this is the horizontal length of the edge in the Newton polygon of $f$ with slope $-w$, as described in Example~\ref{ex:first-trop}.
\end{example}
This example shows how one of Baker and Lorscheid's theorems \cite[Theorem~D]{BL} is a special case of ours and this will also provide an alternative proof of their theorem which we will see in Section~\ref{sec:connections}.

\begin{example}
    The ordinary generating function of the Catalan numbers satisfies the equation $f(C) = 1 - C + xC^2 = 0$. This is a polynomial in $C$ with coefficients in $\R[x] \subset \R((x))$. The two initial forms of $f$ are $\In_0 f = 1 - C$ and $\In_{-1} f = -C + C^2 = -C(1 - C)$.

    The initial form $\In_0 f$ has one positive root and therefore Theorem~\ref{thm:split} tells us to expect one positive root with valuation $0$. Likewise, $\In_{-1} f = -C + C^2$ has one positive root and therefore we should also expect one positive root with valuation $-1$. This all agrees with the explicit solutions we can compute:
    \[
        C_1 = 1 + x + 2x^2 + 5x^3 + \cdots, C_2 = \frac1x - 1 - x - 2x^2 - \cdots. \qedhere
    \]
\end{example}

In Section~\ref{sec:connections}, we show that tropical extension preserves the property of having $\sum_{b \in B} \mult^B_b f$ be bounded by $\deg f$ for all polynomials $f \in B[x]$. This gives some partial understanding to a question asked by Baker and Lorscheid about which hyperfields have this property.

\begin{definition} \label{def:degreebound}
    We say that a whole idyll $B$ \emph{satisfies the degree bound} if for every polynomial $f \in B[x]$,
    \[
        \sum_{b \in B^\bullet} \mult^B_b f \leq \deg f.    
    \]
\end{definition}

\begin{restatable}{maintheorem}{degreeboundthm} \label{thm:degreebound}
    If $B$ satisfies the degree bound and $C \in \Ext^1(\Gamma, B)$ then $C$ satisfies the degree bound.
\end{restatable}

Finally, Bowler and Su have a classification of stringent hyperfields \cite[Theorem~4.10]{BS}. A hyperfield is \emph{stringent} if every sum $a \boxplus b$ is a singleton unless $b = -a$. Bowler and Su's classification says that a hyperfield is stringent if and only if it is a tropical extension of a field, of $\K$, or of $\S$. This gives us the following corollary.
\begin{restatable}{maincorollary}{stringentcor} \label{cor:stringent}
    Every stringent hyperfield satisfies the degree bound.
\end{restatable}

By \cite[Proposition~B]{BL}, a corollary of this degree bound is that if $\vf : K \to C$ is a morphism from a field $K$ to $C$, then
\[
    \mult^C_c f = \sum_{a \in \vf^{-1}(c)} F
\]
for all polynomials $f \in C[x]$. In particular, this is true for every stringent hyperfield (Corollary~\ref{cor:degree-bound}).

\subsection{Relationship to other papers}
There are two papers which have a close relationship with this one. First is the author's previous paper \cite{G} which proves Theorems~\ref{thm:split} and \ref{thm:lifting} but only for the real tropical hyperfield $\TR = \S[\R]$. The current paper was developed in the editing and revision process for that paper and generalizes the previous paper.

Specifically, here we consider tropical extensions of any rank as well as extensions of any (whole) idyll, not just the extension $\S \to \TR$. Theorems~\ref{thm:split}, \ref{thm:main}, \ref{thm:lifting} generalize one of the main theorems of this previous paper~\cite[Theorem~A]{G}. Theorem~\ref{thm:degreebound} and Corollary~\ref{cor:stringent} are entirely new to this paper.
On the other hand, there are a few things covered in the first paper but not in the current one:
\begin{enumerate}
    \item The first paper spends more time discussing properties of $\TR$ and what it means to have a morphism from a field $K$ to $\TR$ (i.e.\  to have a compatible valuation and total order on $K$)~\cite[Section~2.2.1]{G}.
    \item The first paper gives a proof of the multiplicity formula for fields with a morphism to $\TR$ in the language of fields---in particular without using the result for the hyperfield $\TR$~\cite[Section~3]{G}.
    \item There is a weak lifting theorem given a polynomial over $\TR$ to a polynomial over the field of Hahn series $\R[[t^{\R}]]$ having the same number of roots whose leading coefficient is real and positive~\cite[Theorem~5.4]{G}.
\end{enumerate}

The second paper that has close similarities is that of Marianne Akian, Stéphane Gaubert and Hanieh Tavakolipour~\cite{AGT}. They also consider more general tropical extensions than just $\S \to \TR$. In their paper, they work with a type of algebras introduced by Rowan, called \emph{semiring systems} \cite{R}. These bear some similarities to ordered blueprints but the translation is opaque. Akian, Gaubert and Rowan give some comments about the differences and similarities~\cite[Section 5.1]{AGR} but no direct translation has yet been described.
Both frameworks have interest, as well as different connections and potential future development.

Within Rowan's framework, Akian, Gaubert, and Tavakolipour prove a version of Theorems~\ref{thm:split} and \ref{thm:main} using similar techniques (initial forms) \cite[Theorem~5.11]{AGT}.
Their paper also greatly extends the weak lifting theorem \cite[Theorem~5.4]{G} by proving that multiplicities over semiring systems analogous to the idylls $\S[\Gamma]$ can be lifted to any real closed field~\cite[Theorem~7.8]{AGT} (and the roots are real rather than just having a series whose leading coefficient is real).

\subsection{Acknowledgements}
I thank the anonymous reviewer of the previous paper for their comments. I thank Matt Baker for his suggestions and support in finishing this paper. I thank Oliver Lorscheid for correcting some definitions. I thank James Maxwell for some discussion of an early draft of this paper. I thank Marianne Akian, Stéphane Gaubert and Hanieh Tavakolipour for explaining some of the literature on semirings with a symmetry and semiring systems as well as other helpful comments.

I again thank all the people who talked with me about the ideas in the previous paper and/or for suggesting edits: Matt Baker, Marcel Celaya, Tim Duff, Oliver Lorscheid, Yoav Len and Josephine Yu.

\section{Idylls and Ordered Blueprints} \label{sec:idylls}
Several interrelated ring-like and field-like algebraic theories have been used as a framework for $\F_1$ geometry \cite{LF1}, matroid theory \cite{BB, BL2}, and polynomial multiplicities \cite{BL} among other uses. One such field-like algebra is a \emph{hyperfield} which is a field in which the sum of two elements is a nonempty set. Hyperfields have a lot of axioms which largely mirror their classical counterparts and these are described in \cite{BL} as well as the author's previous paper \cite{G}.

In this paper, we work with a generalization of fields and hyperfields called \emph{idylls}, which also have simpler axioms than their hyperfield counterparts. To start, fix a monoid $B^\bullet$ which one can think of as the multiplicative structure of a ring.

\begin{definition}
    For us, monoids have two distinguished elements: $0$ and $1$, and the following axioms:
    \begin{itemize}
        \item multiplication is commutative and associative,
        \item $1$ is a unit: $1 \cdot x = x$ for all $x$,
        \item $0$ is absorbing: $0 \cdot x = 0$ for all $x$.
    \end{itemize}
    These are also called \emph{pointed monoids} or \emph{monoids-with-zero} in the literature.
\end{definition}

Second, we form the free semiring, $\N[B^\bullet]$, which is a quotient of the semiring of finitely supported formal sums by the ideal $\langle 0 \rangle$. I.e.
\[
    \N[B^\bullet] \coloneqq \frac{\left\{\sum_{i = 1}^n x_i : x_i \in B^\bullet\right\}}{\{0, 0 + 0, 0 + 0 + 0, \dots\}}.
\]

\idylldef*

\begin{remark} \label{rem:tract-v-idyll}
    For some purposes, it is enough to just assume that $N_B$ is closed under multiplication by elements in $B^\bullet$ rather than requiring it to be an ideal. Such algebras are called \emph{tracts} and are used, for instance, in the work of Baker and Bowler on matroids \cite{BB}.
\end{remark}

\subsection{Ordered Blueprints}
One reason to prefer the stronger notion of idylls is because these are special cases of Oliver Lorscheid's theory of ordered blueprints \cite{L, LF1, LGB1, LGB2, LSTT}. We can think about these null-ideals as a ``1-sided relation'' $0 \leqslant x_1 + \dots + x_n$ where we say which sums should be ``null'' (although this does not mean we are considering the quotient by $N_B$). Ordered blueprints expand this 1-sided relation by allowing any preorder on $N[B^\bullet]$ which is closed under multiplication and addition. The next definition makes this precise.

\begin{definition} \label{def:ob}
    An \emph{ordered blueprint} $B = (B^\bullet, \leqslant)$ consists of an underlying monoid $B^\bullet$ and a preorder or \emph{subaddition} $\leqslant$ on $N[B^\bullet]$ satisfying for all $a, b \in B^\bullet$ and $w, x, y, z \in \N[B^\bullet]$
    \begin{itemize}
        \item $a \leqslant a$ \hfill (reflexive on $B^\bullet$)
        \item $a \leqslant b$ and $b \leqslant a$ implies $a = b$ \hfill (antisymmetric on $B^\bullet$)
        \item $x \leqslant y$ and $y \leqslant z$ implies $x \leqslant z$ \hfill(transitive)
        \item $x \leqslant y$ implies $x + z \leqslant y + z$ \hfill (additive)
        \item $w \leqslant x$ and $y \leqslant z$ implies $wy \le xz$ \hfill (multiplicative)
        \item $0 \leqslant \text{(empty sum)}$ and $\text{(empty sum)} \leqslant 0$
    \end{itemize}
    The notation $x \in B$ means $x \in \N[B^\bullet]$.
\end{definition}

\begin{remark} \label{rem:axiomatic-blueprints}
    The relation $\leqslant$ is only necessarily antisymmetric on $B^\bullet$. If we identify all formal sums $x, y \in \N[B^\bullet]$ such that $x \leqslant y$ and $y \leqslant x$, the quotient is a new semiring which is denoted $B^+$. The subaddition $\leqslant$ descends to a partial order (antisymmetric) on $B^+$ \cite[Section~5.4]{L}.

    An equivalent theory of ordered blueprints can be described in terms of a partial order on a semiring $B^+$ which is generated by $B^\bullet$---meaning $B^+$ is a quotient of $\N[B^\bullet]$. In this paper, most of our generating relations will be of the form $0 \leqslant y$ so our preorders will usually be partial orders as well.
    
    Baker and Lorscheid define idylls as a partial order on $B^+$ and take as an axiom that the quotient map $\N[B^\bullet] \to B^+$ is a bijection \cite[Definition~2.18]{BL2}. Since idylls are the primary class of objects for us, we will work with $\N[B^\bullet]$ directly.
\end{remark}

\begin{definition}
    We say that a preorder is \emph{generated by} a collection $S = \{x_i \leqslant y_i : i \in I\}$ if it is the smallest preorder containing $S$.
\end{definition}

\begin{remark}
    There is a natural embedding of the category of idylls in the category of ordered blueprints by letting $\leqslant$ be the preorder generated by $N_B$. I.e. generated by $0 \leqslant \sum a_i$ for every $\sum a_i \in N_B$. We will use both ideal and preorder notation in what follows.
\end{remark}

We can also generalize idylls to a larger (but still proper) subcategory of ordered blueprints. If idylls are field-like then \emph{idyllic} ordered blueprints are their ring-like cousins. We do not assume that idyllic ordered blueprints have additive inverses for the sole reason that it allows us to call $\F_1$---which we will define shortly---an idyllic blueprint.

\begin{definition}
    An ordered blueprint is \emph{idyllic} if it is generated by relations of the form $0 \leqslant \sum x_i$. The \emph{idyllic part} of an ordered blueprint $B$, is an ordered blueprint $B^\idyll$ obtained by restricting to the relation generated by relations of the form $0 \leqslant \sum x_i$ in $B$.
    We will sometimes drop the word ``ordered'' to be less wordy and simply call these ``idyllic blueprints.''

    If $B$ is idyllic, we will again call $N_B = \{\sum x_i \in \N[B^\bullet] : 0 \leqslant \sum x_i\}$ the \emph{null-ideal} of $B$ and call $\leqslant$ a \emph{subaddition}. This gives us a common language to talk about idylls and idyllic blueprints.
\end{definition}

\begin{remark}
    There are other notions of positivity for ordered blueprints. In the language of a partial order on a semiring $B^+$ (Remark~\ref{rem:axiomatic-blueprints}), an ordered blueprint is \emph{purely positive} if it is generated by relations of the form $0 \leqslant \sum x_i$ \cite[Definition~2.18]{BL2}. In this language, an idyllic ordered blueprint is a purely positive ordered blueprint for which the map $\N[B^\bullet] \to B^+$ is a bijection.
\end{remark}

\begin{remark} \label{rem:two-embeddings}
    There are two ways to embed a ring or a field $R$ into the category of ordered blueprints. For both embeddings, the underlying monoid is $R^\bullet = (R, 0_R, 1_R, \cdot_R)$. First, we can embed $R$ as $R^\oblpr$ where the relation is given by $\sum x_i \leqslant \sum y_i$ if the evaluation of those sums in $R$ are equal. The second embedding is $R^\idyll = (R^\oblpr)^\idyll$, where we restrict to relations of the form $0 \leqslant \sum y_i$.
\end{remark}

In other words, there is one embedding into the category of ordered blueprints and a different embedding into the category of idyllic blueprints.

\begin{example} \label{ex:F1}
    The ordered blueprint $\F_1$ has $\F_1^\bullet = \{0, 1\}$ as its underlying monoid, with the usual multiplication. The relation on $\F_1$ is equality (equivalently it has an empty generating set).
    $\F_1$ is the initial object in the category of idyllic ordered blueprints and its null-ideal is $\{0\}$.
\end{example}

\begin{example} \label{ex:krasner}
    The \emph{Krasner idyll} $\K$ on $\{0, 1\}$ is the idyll with null-ideal $\N[\B^\bullet] \setminus \{1\} = \{1 + 1, 1 + 1 + 1, \dots\}$. The Krasner idyll is the terminal object in the category of idylls.
\end{example}

\begin{definition}
    If $R$ is a ring or field and $G$ is a subgroup of $R^\times$, then there is an idyll on $G^\bullet = G \cup \{0\}$ which we call a \emph{partial field idyll}. The null-ideal of $G$ is the set of formal sums whose image in $R$ is zero.
\end{definition}

These are called ``partial'' fields because the sum of two elements of $G^\bullet$ may or may not be in $G^\bullet$ as well. Therefore addition is only partially defined.

\begin{example} \label{def:F1pm}
    As an example, consider $G = \{1, -1\} = \Z^\times$ then $N_G = \langle 1 + (-1) \rangle$. This is called the regular partial field (idyll) and is denoted either $\F_{1^2}$ or $\F_1^{\pm}$ in the literature.

    More generally, if $G = \mu_n \subset \C^\times$ is the group of $n$-th roots of unity, then $G$ forms an idyllic ordered blueprint called $\F_{1^n}$. This lacks additive inverses if $n$ is odd, but is still field-like.
\end{example}

\begin{definition} \label{def:hyperfield-idyll}
    If $K$ is a field and $G$ is a subgroup of $K^\times$ then there is am idyll on $K / G$ called a \emph{hyperfield idyll}. The null-ideal $N_{K/G}$ is the set of sums of equivalence classes $[a_1] + \dots + [a_n]$ for which there exists representatives $x_i \in [a_i], i = 1,\dots,n$ for which $x_1 + \dots + x_n = 0$ in $K$.
\end{definition}

\begin{remark} \label{ex:sumset}
    Partial fields and hyperfields are also thought about in terms of the sum sets $a \boxplus b \coloneqq \{c : a + b - c \in N_B\}$. For partial fields, every sum is either empty or a singleton. For hyperfields, every sum is non-empty. Fields, therefore, are the intersection of partial fields and hyperfields.
\end{remark}

\begin{remark} \label{rem:hyperfield-not-nec-quots}
    All of the hyperfield (idyll)s named in this paper are quotients of some field by a multiplicative group, however there are hyperfields which are not of this form. Christos Massouros was the first to construct an example of such \cite{Mas}. For the purposes of this paper, it is sufficient to think only about quotient hyperfields.

    For quotients, the sum sets defined in the previous remark can also be defined for an arbitrary number of summands by
    \[\bigboxplus_i [a_i] = \{a_i' : a_i' \in [a_i]\}. \]

    If the hyperfield is not a quotient, then we need to define repeated hyperaddition monadically:
    \begin{itemize}
        \item identify $a$ with $\{a\}$,
        \item flatten sums, so \[ a_0 \boxplus (a_1 \boxplus \dots \boxplus a_n) = \bigcup \{ a_0 \boxplus t : t \in a_1 \boxplus \dots \boxplus a_n \}. \]
    \end{itemize}
    Defining hyperfields as idylls (Definition~\ref{def:hyperfield}) skips having to talk about the monadic laws.

    The monadic laws are closely related to a \emph{fissure rule} (Remark~\ref{rem:fissure-rule}) for pastures, which we will define later on. We will make use of in Proposition~\ref{prop:pasture-equivalence}. We will also reference the monadic laws in Section~\ref{sec:hyperfields} where we define hyperfields as special kinds of pastures.
\end{remark}

\begin{example}
    As an example, the Krasner idyll $\K$ is a quotient $K/K^\times$ for any field $K$ other than $\F_2$. The idyll of signs $\S$ is the quotient $\R / \R_{> 0}$. The tropical idyll is the quotient of a valued field $K$ with value group $\R$ by the group of elements with valuation $0$.

    More generally, if the value group of $K$ is any ordered Abelian group $\Gamma$, then the same quotient $K / v^{-1}(0)$ gives an idyll structure on $\Gamma$ which we will see again in Definition~\ref{def:OAGidyll}. This is \emph{a} tropical idyll but not \emph{the} tropical idyll---a term reserved for $\Gamma = \R$. Instead, we will call these \emph{OAG idylls}.
\end{example}

\begin{example}
    The \emph{idyll of phases} or \emph{phase idyll} $\P$, is the hyperfield idyll on the quotient $\C / \R_{> 0}$.
    A sum of phases $\sum e^{i\theta_k}$ belongs to $N_\P$ if there are magnitudes $a_k \in \R_{> 0}$ for which $\sum a_k e^{i\theta_k} = 0$. Equivalently, we can define the null-ideal using convex hulls as
    \[
        N_{\P} = \left\{ \sum_k e^{i\theta_k} : 0 \in \operatorname{int}\left(\conv_k\left(e^{i\theta_k}\right) \right) \right\}
    \]
    where $\operatorname{int}(\conv_k(e^{i\theta_k}))$ is the interior of the convex hull relative to its dimension. E.g.\ if the convex hull is a line segment, then the interior is a line segment without the endpoints.
\end{example}

\begin{remark}
    As Bowler and Su point out in a footnote \cite[page 674]{BS}, there are actually two phase hyperfields: the quotient $\P = \C / \R_{> 0}$ as defined above, and the one that Viro originally defined \cite{V1}. The difference is in whether you require $0$ be in the interior of $\conv_k(e^{i\theta_k})$ (our definition) or if it is allowed to lie on the boundary (Viro's definition).
\end{remark}

\subsection{Morphisms of Ordered Blueprints}
\begin{definition} \label{def:morphism}
    If $B, C$ are two ordered blueprints, a (homo)morphism $f : B \to C$ consists of a morphism of monoids $f^\bullet : B^\bullet \to C^\bullet$ such that the induced map $f : B \to C$ is order-preserving.
    In particular, for all $x, y, x_i, y_i \in B^\bullet$
    \begin{itemize}
        \item $f^\bullet(xy) = f^\bullet(x) f^\bullet(y)$
        \item $f^\bullet(0_B) = 0_C$
        \item $f^\bullet(1_B) = 1_C$
        \item if $\sum x_i \leqslant \sum y_i$ then $f(\sum x_i) \leqslant f(\sum y_i)$
    \end{itemize}
\end{definition}

\begin{definition}
    A morphism of idylls is a morphism of their corresponding ordered blueprints. I.e.\ it is a morphism of monoids such that $f(N_B) \subseteq N_C$.
\end{definition}

\subsubsection{Valuations}
Classically, a (rank-$1$) valuation on a field $K$ is a map $v : K \to \R \cup \{\infty\}$ such that
\begin{itemize}
    \item $v(0) = \infty$,
    \item $v$ restricts to a group homomorphism $F^\times \to (\R, +)$,
    \item and for every $a, b \in F$, we have $v(a + b) \ge \min\{v(a), v(b)\}$.
\end{itemize}
In our language, a valuation in this sense is simply a morphism from a field $K$ (viewed as an idyll) to the tropical idyll $\T$.

We can also substitute $\R$ with any ordered Abelian group (OAG).

\begin{definition}
    An \emph{ordered Abelian group} (OAG) is an Abelian group $(\Gamma, +)$ with a total order $\le$ for which $a \le b$ implies $a + c \le b + c$ for all $a, b, c$.
    The \emph{rank} of an OAG is $\dim_\R \Gamma \otimes \R$ and the phrase ``higher-rank'' simply means that the rank is at least $2$.
\end{definition}

\begin{example} \label{ex:lex-order}
    On $\R^n$, there is a \emph{lexicographic} or \emph{dictionary} order $\lex$ defined inductively by $(a_1, \dots, a_n) \lex (b_1, \dots, b_n)$ if either $a_1 < b_1$ or $a_1 = b_1$ and $(a_2, \dots, a_n) \lex (b_2, \dots, b_n)$. We will make use of this order in Section~\ref{sec:higher-rank}
\end{example}

\begin{definition} \label{def:OAGidyll}
    The idyll $\Gamma^\idyll$ is the idyll on $\Gamma^\bullet = \Gamma \cup \{\infty\}$ where
    \begin{itemize}
        \item $\infty$ is the absorbing element
        \item $0$ is the unit element (writing things additively)
        \item $\sum a_i \in N_\Gamma$ if and only if the minimum term appears at least twice
    \end{itemize}
    We will call this an \emph{OAG idyll}.
\end{definition}

\begin{definition}
    In our framework, a \emph{valuation} $v$ on an idyllic ordered blueprint $B$, is a morphism $v : B \to \Gamma^\idyll$ for some ordered Abelian group $\Gamma$. The letter $v$ will be reserved for a valuation of some kind and usually for the valuation $C \to \Gamma^\idyll$ which appears in the definition of a tropical extension.
\end{definition}

Now we will check that valuations as we have just defined, agree with the usual notion of a Krull valuation as well as illustrate some properties of valuations.

\begin{proposition}
    If $R$ is a ring and $v : R^\idyll \to \Gamma^\idyll$ is a valuation, then
    \begin{enumerate}[label=(V\arabic*)]
        \item \label{prop:val:monoid} $v^\bullet : R^\bullet \to \Gamma^\bullet$ is a monoid homomorphism,
        \item \label{prop:val:zero} $v^\bullet(0_R) = \infty$,
        \item \label{prop:val:rootsofunity} if $u^n = 1_R$ for some $n \ge 1$ then $v^\bullet(u) = 0$,
        \item \label{prop:val:ultrametric} $v^\bullet(a +_R b) \ge \min\{v^\bullet(a), v^\bullet(b)\}$ for all $a, b \in R$.
        \item \label{prop:val:distinct} if $v^\bullet(a) \neq v^\bullet(b)$ then $v^\bullet(a +_R b) = \min\{v^\bullet(a), v^\bullet(b)\}$.
    \end{enumerate}
    Conversely, a map $v^\bullet : R^\bullet \to \Gamma^\bullet$ with these properties induces a valuation $v : R^\idyll \to \Gamma^\idyll$.
\end{proposition}

\begin{proof}
    Properties \ref{prop:val:monoid} and \ref{prop:val:zero} follow by definition and Property \ref{prop:val:rootsofunity} follows from Property \ref{prop:val:monoid}.

    For Property \ref{prop:val:ultrametric}, if $c = a +_R b$ in $R$, then $0 \leqslant_R a + b - c$ in $R^\idyll$. Therefore, $v(a + b - c) = v^\bullet(a) + v^\bullet(b) + v^\bullet(c) \in N_\Gamma$ and this, by definition, means that the minimum of $v^\bullet(a), v^\bullet(b), v^\bullet(c)$ occurs at least twice. It is impossible therefore, to have $v^\bullet(c) < \min\{v^\bullet(a), v^\bullet(b)\}$.

    Property \ref{prop:val:distinct} follows because if the minimum of $v^\bullet(a), v^\bullet(b), v^\bullet(a +_R b)$ needs to occur at least twice and $v^\bullet(a) \neq v^\bullet(b)$, then $v^\bullet(a +_R b)$ must be equal to the minimum of $v^\bullet(a), v^\bullet(b)$.

    Conversely, suppose $v^\bullet$ satisfies these properties and $0 \leqslant \sum x_i$ in $R^\idyll$---meaning $\sum_R x_i = 0_R$ in $R$ and we may assume that at least one of the $x_i$'s are nonzero or else there is nothing to show. Given this, we know that the minimum of the quantities $v^\bullet(x_i)$ occurs at least twice because otherwise $v^\bullet(\sum_R x_i) = \min v^\bullet(x_i)$ by property \ref{prop:val:distinct}. But $v^\bullet(\sum_R x_i) = v^\bullet(0_R) = \infty \neq \min \{v^\bullet(x_i)\}$. So we conclude that the minimum occurs at least twice and hence $0 \leqslant \sum v^\bullet(x_i)$ in $\Gamma^\idyll$.
\end{proof}

\begin{remark}
    A more general definition of valuations exists where the source and target can be any ordered blueprint \cite[Section 3]{LSTT}, \cite[Chapter~6]{L}.
\end{remark}

\subsection{Images, Equalizers and Subblueprints} \label{sec:cat-constructions}
We will now define a few categorical constructions which are useful in our constructions. Particularly for describing tropical extensions.

\begin{definition} \label{def:subblueprint}
    A \emph{subblueprint} $B$ of an ordered blueprint $C$ is a submonoid $B^\cdot \subseteq C^\bullet$ and such that if $\sum x_i \leqslant \sum y_i$ in $B$ then $\sum x_i \leqslant \sum y_i$ in $C$. The subblueprint is \emph{full} if the converse holds: if $\sum x_i$ and $\sum y_i \in B$ then $\sum x_i \leqslant \sum y_i$ in $C$ if and only if $\sum x_i \leqslant \sum y_i$ in $B$.
\end{definition}

\begin{remark}
    A full subblueprint is determined entirely by the submonoid $B^\bullet \subseteq C^\bullet$ and we will call this an \emph{induced} subblueprint.
\end{remark}

\begin{definition}
    If $f : B \to C$ is a morphism of ordered blueprints, its \emph{image} is the subblueprint $\im(f)$ on the monoid $\im(f)^\bullet = f(B^\bullet) \subseteq C^\bullet$ where $\sum f^\bullet(x_i) \leqslant \sum f^\bullet(y_i)$ in $\im(f)$ if and only if $\sum x_i \leqslant \sum y_i$ in $B$.
\end{definition}

\begin{definition}
    Given two maps $f, g : B \to C$, their \emph{equalizer} $\eq(f, g)$ is the induced subblueprint of $B$ on the monoid $\eq(f,g)^\bullet = \{x \in B^\bullet : f(x) = g(x)\}$.
\end{definition}

\begin{definition}
    If $v : C \to \Gamma^\idyll$ is a valuation on an idyll $C$, we define a morphism $1 : C \to \Gamma^\idyll$ by $1^\bullet(x) = 1_{\Gamma^\idyll}$ if $x \neq 0_C$ and $1^\bullet(0_C) = 0_{\Gamma^\idyll}$. This is a morphism because idylls have proper null-ideals, meaning if $\sum x_i \in N_C$ then there are at least two nonzero $x_i$'s and so $1\left( \sum x_i \right) \in N_\Gamma$ since the minimum occurs at least twice.
\end{definition}

We can also describe the morphism $1$ as the composition of the sequence $C \xrightarrow{v} \Gamma^\idyll \to \K \to \Gamma^\idyll$.

\section{Polynomial and Tropical Extensions} \label{sec:extensions}
Let us turn our attention next to generalizing polynomial rings to polynomials over idylls. Remember that additive relations in idylls are handled by an ideal in some free semiring. The terms in those additive relations form a monoid. This suggests the following definition.

\begin{definition} \label{def:poly-ext}
    Let $B$ be an idyll with monoid $B^\bullet$ and null-ideal $N_B \subset \N[B^\bullet]$. The \emph{polynomial extension} of $B$ is an idyllic ordered blueprint, which we call $B[x]$. Its underlying monoid is
    \[
        B[x]^\bullet = \{bx^n : b \in B^\bullet, n \in \N\} / \langle 0x^n \equiv 0 : n \in \N \rangle
    \]
    with multiplication given by $(bx^m)(cx^n) = (bc)x^{m + n}$. The null-ideal of $B[x]$ is the ideal in $\N[B[x]^\bullet]$ which is generated by $N_B$.
\end{definition}

\begin{definition} \label{def:poly}
    When we say a \emph{polynomial}, we mean that which might otherwise be called a \emph{pure polynomial}. A (pure) polynomial is an element of $\N[B[x]^\bullet]$ for which there is at most one term in each degree. E.g.\ $x + x^2 + x^5$ is a polynomial but $x + x + x$ is not.
\end{definition}

\begin{remark}
    These polynomial extensions are free objects in the category of $B$-algebras---where an algebra over $B$ is a morphism $B \to C$. For hyperfields, one can try to define a polynomial ring by extending addition and multiplication:
    \[
        \left(\sum a_i x^i\right) \boxplus \left(\sum b_i x^i \right) = \left\{ \sum c_i x^i : c_i \in a_i \boxplus b_i \right\}
    \]
    and
    \[
        \left(\sum a_i x^i\right) \boxdot \left(\sum b_j x^j \right) =
        \left\{ \sum c_k x^k : c_k \in \bigboxplus_{i + j = k} a_i b_j \right\}.
    \]
    This creates an algebra in which both addition and multiplication are multivalued. Unfortunately, this naïve definition fails to be free and, more egregiously, the multiplication is often not associative either \cite[Appendix~A]{BL}.
\end{remark}

A related construction to polynomial extensions is that of a split tropical extension.

\begin{definition} \label{def:split-extension}
    Let $B$ be an idyll and let $\Gamma$ be an OAG. Form the pointed group
    \[
        B[\Gamma]^\bullet = \{bt^\gamma : b \in B^\bullet, \gamma \in \Gamma\} / \langle 0t^\gamma \equiv 0 : \gamma \in \Gamma \rangle
    \]
    with multiplication given by $(b_1t^{\gamma_1})(b_2t^{\gamma_2}) = (b_1b_2) t^{\gamma_1 \gamma_2}$.

    The null-ideal of $B[\Gamma]$ is the set of all formal sums $\sum a_it^{\gamma_i}$ such that if we let $I = \{i : \gamma_i \text{ is minimum}\}$ then $\sum_I a_i \in N_B$.
\end{definition}

Split tropical extensions come with a natural valuation map $v : B[\Gamma] \to \Gamma^\idyll$ given by $v^\bullet(bt^\gamma) = \gamma$. For split tropical extensions, there is a splitting $\Gamma \to B[\Gamma]^\times$ given by $\gamma \mapsto t^\gamma$.

\begin{remark}
    Going forward, we will often drop the `$t$' from the notation and simply write $b^\gamma$ instead of $bt^\gamma$ and $1^\gamma$ instead of $t^\gamma$. This helps avoid confusing $B[\Gamma]$ with a polynomial extension since $B[\Gamma]$ has some additional relations on it beyond those of just polynomials.
\end{remark}

More generally, a tropical extension is any idyll which fits into an exact sequence with $B$ and $\Gamma$ and with similar rules about the null-ideal as for split extensions.

\tropextdef*

\begin{remark}
    From Section~\ref{sec:cat-constructions}, to say that $\im(\iota) = \eq(v, 1)$ means that $0 \leqslant \sum x_i$ in $B$ if and only if $0 \leqslant \sum \iota^\bullet(x_i)$ in $\eq(v, 1) \subseteq C$. I.e.\ $\im(\iota)$ is a full subblueprint of $C$. Because of this, we can safely make the assumption that $B^\bullet \subseteq C^\bullet$ and $\iota$ is the identity.
\end{remark}

\begin{remark}
    Tropical extensions of idylls are closely related to tropical extensions for semiring with a symmetry \cite{AGG} or for semiring systems \cite{R, AGT}. For hypergroups and (skew) hyperfields, tropical extensions appear as a semidirect product in the work of Bowler and Su~\cite{BS}.
\end{remark}

\begin{remark}
    Tropical extensions, have ``levels'' $B^\gamma = \{c \in C : v^\bullet(c) = \gamma\}$ which are not-necessarily-canonically isomorphic to $B^\times$ and $B^0$ is canonically isomorphic to $B^\times$. The relations on $B^\gamma$ are uniquely determined by the torsor action $B^0 \times B^\gamma \to B^\gamma$. (See also Section~\ref{sec:non-split-case}.)

    Additionally, to say that a relation $\sum a_i \in N_C$ holds if and only if it holds among the minimal valuation terms, means that if we have a sum like $a - a \in N_C$ then $a - a + b \in N_C$ for any element $b$ of larger valuation. In other words, the sum set $a \boxplus (-a)$ from Example~\ref{ex:sumset} contains every element whose valuation is strictly larger than $v^\bullet(a)$.
    
    These properties about levels and sum sets are the basis for how Bowler and Su describe their semidirect product.
    We will give a formal proof of this equivalence in Section~\ref{sec:hyperfields}.
\end{remark}

\begin{remark}
    By Bowler and Su's classification \cite[Theorem~4.17]{BS}, if $B$ is either $\K$ or $\S$ then every tropical extension by $B$ is split.
\end{remark}

\begin{example}
    The tropical idyll $\T = \K[\R]$ is a split tropical extension of $\R$ by $\K$. The only caveat is a slight change of notation: we defined elements of $\T^\times$ as real numbers but we defined elements of $\K[\R]^\times$ as being of the form $1^\gamma$ where $\gamma$ is a real number.

    For instance, the sum $0 + 0 + 1$ in $N_\T$ corresponds to $1^0 + 1^0 + 1^1$ in $\K[\R]$. This is in $N_{\K[\R]}$ because if we take the sum of the coefficients of the minimum terms, we get $1 + 1 \in N_\K$.
\end{example}

\begin{example}
    Every OAG idyll is a tropical extension in a natural way: $\Gamma^\idyll = \K[\Gamma]$ (again with a change of notation). For example, we have higher-rank tropical idylls such as $\T_n \coloneqq (\R^n, \lex)^\idyll = \K[\R^n]$. Moreover, there is a natural isomorphism $\T_m[\R^n] = \T_{m + n}$ (Example~\ref{ex:tropext-firstexamples}).
\end{example}

\begin{example} \label{ex:trop-reals}
    Extensions by $\S$ give signed tropical extensions. For instance, $\TR = \S[\R]$ is the tropical real idyll/hyperfield which was first introduced by Oleg Viro~\cite{V2}.

    The null-ideal of $\TR$ is given by sums where the minimum terms appear at least twice and with at least one positive and one negative term among them. E.g. $t + (-1)t + t^2$ has one positive minimum term, $t$, and one negative minimum term, $(-1)t$.
\end{example}

\begin{example}
    Extensions by $\P$ give phased tropical extensions. For example, $\TP = \TC = \P[\R]$ is the tropical phase idyll or tropical complex idyll. This was also introduced as a hyperfield by Viro (\emph{ibid.}).
\end{example}

\begin{remark} \label{rem:not-pullback}
    For the tropical reals, there is a map $\operatorname{sign} : \TR \to \S$ which give the sign of the leading coefficient. It is tempting to think that $\TR$ is isomorphic to the pullback
    \begin{center}
        \begin{tikzcd}
            \S \times \T \arrow[r] \arrow[d] & \S \arrow[d] \\
            \T \arrow[r]                     & \K          
        \end{tikzcd}
    \end{center}
    but this is not the case.
    In $\S \times \T$, we have the relation $0 \leqslant 1^0 + 1^0 + (-1)^1$ because its images in $\S$ and $\T$ are relations. However, this is not a relation in $\TR$ since among the terms of minimal valuation, they are all positive.

    See \cite[Section~5.5]{L} for a discussion of various (co)limits in the category of ordered blueprints.
\end{remark}

\subsection{Newton Polygons and Initial forms}
\subsubsection{Newton Polygons} \label{sec:newt-poly}
Associated to polynomials over a tropical extension or over a valued field, is an object called the Newton polygon. To define this, we require a rank-$1$ valuation $v : B \to \T$.

\begin{definition}
    We define a \emph{lower inequality} on $\R^2$ to be an inequality of the form $\langle u, x \rangle \ge c$ for some $c \in \R$ and some $u$ is in the upper half plane: $u \in \{ (u_1, u_2) : u_2 \ge 0 \}$. Every lower inequality creates a halfspace $H(u, c) = \{x : \langle u, x \rangle \ge c\}$.

    Given a set of points $S \subset \R^2$, its \emph{Lower Convex Hull} is defined as the intersection of the halfspaces containing $S$ where $u = (u_1, u_2)$ is in the upper half plane:
    \[
        \LCH(S) = \bigcap \{ H(u, c) : S \subseteq H(u, c), u_2 \ge 0 \}.
    \]
\end{definition}

\begin{definition}
    Let $v : B \to \T$ be a valuation on $B$ and let $f \in B[x], f = \sum_I b_i x^i$ be a polynomial. The \emph{Newton polygon} of $f$ is
    \[
        \newt(f) = \LCH\left( \{(i, v^\bullet(b_i)) : i \in I \} \right).
    \]
    Additionally, by an \emph{edge} of the Newton polygon, we will always mean a bounded edge.
\end{definition}

\begin{example} \label{ex:newt-poly}
    Consider the polynomial $f = 2 + 1x + 0x^2 + 0x^3 + 2x^4 + 1x^5 \in \T[x]$ where $v : \T \to \T$ is the identity. The Newton polygon of $f$ is shown in in Figure~\ref{fig:newt-poly}.
\end{example}

\begin{figure}[htbp]
    \centering

    \begin{tikzpicture}
        \filldraw [fill=black!10, draw=black] (0,2.5) -- (0,2) -- (2,0) -- (3,0) -- (5,1) -- (5,2.5);
        \foreach \x/\y in {0/2, 1/1, 2/0, 3/0, 4/2, 5/1}
            \filldraw (\x,\y) circle (0.5mm);
    \end{tikzpicture}

    \caption{Newton polygon of $f$ in Example~\ref{ex:newt-poly}}
    \label{fig:newt-poly}
\end{figure}
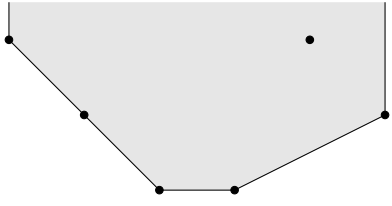

\subsection{Initial Forms}
Now we will define a ``leading coefficient'' and initial form operator for tropical extensions. First, for split extensions, we take the following definition.

\begin{definition} \label{def:split-lc-in}
    For the split extension $B[\Gamma]$, define $\lc^\bullet : B[\Gamma]^\bullet \to B^\bullet$ by $\lc^\bullet(b^\gamma) = b$. This does not induce a morphism of ordered blueprints (c.f.\ Remark~\ref{rem:not-pullback}).

    If $\gamma \in \Gamma$, define $\In_\gamma : B[\Gamma][x] \to B[x]$ by
    \[
        \In_\gamma\left( \sum b_i^{\gamma_i} x^i \right) = \sum_I \lc^\bullet(b_i^{\gamma_i}) x^i
    \]
    where $I = \{i : \gamma_i + i\gamma \text{ is minimal}\}$.
\end{definition}

\begin{example} \label{ex:initial-form}
    Consider the polynomial $f = 2 + 1x + 0x^2 + 0x^3 + 2x^4 + 1x^5 \in \T[x]$ from Example~\ref{ex:newt-poly} whose Newton polygon is shown in Figure~\ref{fig:newt-poly}

    The Newton polygon of $f$ has edges of slope $-1, 0, \frac12$ and the corresponding initial forms are $\In_1 f = 1 + x + x^2, \In_0 f = x^2 + x^3$ and $\In_{-1/2} f = x^3 + x^6 \in \K[x]$. All other initial forms of $f$ are monomials.
\end{example}

\subsubsection{Newton Polygons for Higher-Rank} \label{sec:higher-rank}
Consider a polynomial $f = \sum b_ix^i$ with coefficients in $\T_n = \K[\R^n] = (\R^n, \lex)^\idyll$ where $\lex$ is the lexicographic order from Example~\ref{ex:lex-order}. Or more generally, we could have coefficients in $B$ where $B$ is equipped with a valuation $v : B \to \T_n$. In the previous section, we gave a definition of an initial form $\In_{\gamma}(f)$ which covers this, but the connection to Newton polygons is less clear. To figure out how to define this, we are going to consider a sequence of rank-$1$ valuations using the natural identity $\T_n = \T_{n-1}[\R]$.

Define $v_n : \T_n \to \T$ as the valuation on $\T_{n-1}[\R]$. Explicitly, given $\gamma = (\gamma_1, \dots, \gamma_n) \in (\R^n, \lex)$, we have $v_n(\gamma) = \gamma_1$. Let $\In^{v}_\gamma$ denote the initial form operator with respect to an extension $B[\Gamma] \xrightarrow{v} \Gamma$. So for example, $\In^{v_n}_{\gamma_1}$ means we are considering $\T_n$ as an extension of $\R$ by $\T_{n-1}$ rather than as an extension of $\R^n$ by $\T_0 = \K$. With this, we have the following lemma.

\begin{lemma} \label{lem:inductive-in}
    With the notation as above, we have
    \[
        \In^v_\gamma(f) = \In^{v_1}_{\gamma_n}( \cdots \In^{v_{n-1}}_{\gamma_2}(\In^{v_n}_{\gamma_1}(f))).
    \]
\end{lemma}
Therefore, when we consider a higher-rank valuation, we are thinking about a sequence of Newton polygons rather than one Newton polytope.

\begin{proof}
    This is an inductive statement so to simplify notation, we will use $n = 2$ to illustrate.

    Let $f = \sum (a_i, b_i)x^i \in \T_2[x]$ and let $\gamma = (\lambda, \mu) \in \T_2^\times$. Let $I = \{i : (a_i, b_i) + i(\lambda, \mu) \text{ is minimal}\}$ and let that minimal value be $(\lambda_0, \mu_0)$. Next, let $I_1 = \{i : a_i + i\lambda \text{ is minimal}\}$ and let that minimal value be $\lambda_0'$. First, we claim that $\lambda_0' = \lambda_0$.

    If not, then we must have $\lambda_0' < \lambda_0$ or else $\lambda_0$ would be minimal minimal for $I_1$ as well. Now if $(\lambda_0, \mu_0) = (a_{i_0}, b_{i_0}) + i_0(\lambda, \mu)$ and $\lambda_0' = a_{i_1} + i_1\lambda$, then by definition (Example~\ref{ex:lex-order}), it must be that $\lambda_0' \ge \lambda_0$ or else $(a_{i_1}, b_{i_1}) + i_1(\lambda,\mu) <_{\rm lex} (\lambda_0, \mu_0)$ and this contradicts minimality.

    Now define $I_2 = \{i \in I_1 : b_i + i\mu \text{ is minimal}\}$ and we have likewise that this minimal value is $\mu_0$. I.e.\ we have $I_2 = I$. Putting this together, we have
    \[
        \In^v_{(\lambda, \mu)} = \sum_I x^i = \In^{v_2}_\mu\left( \sum_{I_1} b_ix^i \right) = \In^{v_2}_\mu (\In^{v_1}_{\lambda}(f)). \qedhere
    \]
\end{proof}

If the coefficients were in $B[\R^n]$ rather than $\T_n = \K[\R^n]$, then everything works the same by changing notation from $\sum \gamma_i x^i = \sum (1_{\K})^{\gamma_i} x^i$ to $\sum c_i^{\gamma_i} x^i$.

\begin{remark}
    Lemma~\ref{lem:inductive-in} demonstrates that we can apply Theorem~\ref{thm:split} inductively by considering a sequence of rank-1 extensions.
\end{remark}

\subsubsection{The Non-Split Case} \label{sec:non-split-case}
If $C \in \Ext^1(\Gamma, B)$ is a non-split extension, then the leading coefficient map is no longer well defined because we can no longer simply divide by $t^\gamma$. Instead, for every fixed element $c_0 \in C^\times$ with $v^\bullet(c) = \gamma$, we get a map $\{c \in C^\times : v^\bullet(c) = \gamma\} \to B^\times$ by dividing by $c_0$ and this map depends on the choice of $c_0$, i.e. this is a torsor for $B^\times$.

\begin{definition} \label{def:torsors}
    Let $C \in \Ext^1(\Gamma, B)$ be a tropical extension. Because the sequence $B^\times \to C^\times \to \Gamma$ is exact, there is a natural identification of $B^\times$ with the group $B^0 \coloneqq \{c \in C^\times : v^\bullet(c) = 0\}$.
    More generally, let us define $B^\gamma = \{c \in C^\times : v^\bullet(c) = \gamma\}$ and $B^\infty = \{0_C\}$. This gives a grading $C^\bullet = \bigcup_{\gamma \in \Gamma^\bullet} B^\gamma$ where multiplication is graded: $\cdot : B^{\gamma} \times B^{\gamma'} \to B^{\gamma + \gamma'}$.
    In particular, the pairing $B^0 \times B^\gamma \to B^\gamma$ makes $B^\gamma$ into a $B^0$-torsor.

    We will define the leading coefficient map $\lc^\bullet : C^\bullet \to \bigcup_{\gamma \in \Gamma^\bullet} B^\gamma$ which literally is the identity, but we give a name to this to keep the notation consistent. This also helps remind us that the output is in a specific torsor for $B$.
\end{definition}

So now, instead of having initial forms with coefficients in $B$, the coefficients will be in one of these torsors.

\begin{definition} \label{def:nonsplit-in}
    Let $C \in \Ext^1(\Gamma, B)$ be a tropical extension and let $f = \sum c_ix^i \in C[x]$ be a polynomial. Let $a \in C$ be a root of $f$ with valuation $\gamma_1$ and let $\gamma_0 = \min\{v^\bullet(c_i) + i\gamma_1\}$. We will say that $a$ \emph{corresponds to the line} $\ell = \{\gamma_0 - i \gamma_1 : i \in \N\}$.

    Let $I = \{i : v^\bullet(c_i) = \gamma_0 - i \gamma_1 \}$. We define the initial form with respect to $a$ (rather than with respect to $\gamma_1$) as
    \[
        \In_a(f) = \sum_{i \in I} \lc^\bullet(c_i)(ax)^i \in B^{\gamma_0}[x].
    \]
\end{definition}

\begin{remark} \label{rem:initial-forms-identity}
    For split extensions, we have two initial forms. First, we have $\In_\gamma f \in B[x]$ from Definition~\ref{def:split-lc-in}. Second, we have $\In_{a} f \in B^{\gamma_0}[x]$ from Definition~\ref{def:nonsplit-in}. These two polynomials are related via the natural identification $B^\times = B^0$ and the identity
    \[
        \In_\gamma f = 1^{-\gamma_0} \In_{1^\gamma} f.
    \]

    Additionally, if $a = b^\gamma$, then
    \[
        \In_{b^\gamma} f(x) = \In_{1^\gamma} f(bx).
    \]
\end{remark}

\section{Factoring Polynomials and Multiplicities over Idylls} \label{sec:factoring}
We now investigate factoring and multiplicities. First, we will do this for $B[x]$ and show that these notions are an extension to idylls of the Baker-Lorscheid multiplicity operator for hyperfields. Second, we will define this for $B^{\gamma_0}[x]$ and we will see that all the ways to identify $B^{\gamma_0} \cong B$ lead to the same multiplicities and factors.

\subsection{Roots of Polynomials}
There are two serviceable definitions of what it means for a polynomial to have a root. Classically, we can say that $f(x)$ has a root $a$ if $f(a) = 0$ or if $(x - a) \mid f(x)$. For idylls, we will take the latter as the definition and explain in which context, the two definitions agree.

\begin{definition} \label{def:root}
    Let $f(x) = \sum_{i = 0}^n c_ix^i$ be a polynomial over an idyll $B$ and let $a \in B^\bullet$. We will say that $a$ is a \emph{root} of $f$ if there exists a \emph{factorization} $0 \leqslant f(x) - (x - a)g(x)$ for some polynomial $g(x) = \sum d_i x^i$. I.e.\ if $0 \leqslant c_i - d_{i - 1} + ad_i$ for all $i$ (treating the coefficients as infinite sequences with a finite support).
\end{definition}

\begin{definition} \label{def:curlyeq}
    It will be convenient to define a relation $\preccurlyeq$ by $x \preccurlyeq y$ if $0 \leqslant -x + y$. So we will write factorizations as $f(x) \preccurlyeq (x - a)g(x)$ and $c_i \preccurlyeq d_{i - 1} - ad_i$.
\end{definition}

There is a context in which Definition~\ref{def:root} is equivalent to $0 \leqslant f(a)$, called \emph{pastures}. There are a few equivalent definitions of pastures in the literature, we give one of them here.
\begin{definition}
    An ordered blueprint is \emph{reversible} if it contains an element $\epsilon = \epsilon_B$ such that $\epsilon^2 = 1$, we have the relation $0 \leqslant 1 + \epsilon$, and such that if $a, b \in B^\bullet, x \in \N[B^\bullet]$ then $a \leqslant b + x$ implies $\epsilon b \leqslant \epsilon a + x$. By \cite[Lemma~5.6.34]{L}, $\epsilon$ is unique and so is any additive inverse of $a$ for any $a \in B^\bullet$. As with idylls, we will write $-1$ and $-a$ instead of $\epsilon$ and $\epsilon a$.

    A \emph{pasture} is a reversible ordered blueprint generated by relations of the form $a \leqslant b + c$ with $a, b, c \in B^\times$ as well as the relation $0 \leqslant 1 + (-1)$.
\end{definition}

\begin{remark}
    If $B$ is a pasture, its idyllic part $B^\idyll$ satisfies an axiom known as \emph{fusion} where if $a \in B^\bullet$ and $x, y \in \N[B^\bullet]$ then $0 \leqslant x + a$ and $0 \leqslant y - a$ implies $0 \leqslant x + y$.
\end{remark}

\begin{proof}
    By reversibility, $0 \leqslant x + a$ implies $-a \leqslant x$ and $0 \leqslant y - a$ implies $a \leqslant y$. Adding these together, we have
    \[
        0 \leqslant (-a) + a \leqslant x + y. \qedhere
    \]
\end{proof}

\begin{remark}
    The fusion rule is discussed in detail in a paper of Baker and Zhang \cite{BZ}. It is possible to define a pasture as an idyll generated by three-term relations $0 \leqslant a + b + c$ and fusion (the idyllic part of what we have just defined). Just looking at idylls generated by three-term relations but without the fusion axiom gives a \emph{nonequivalent} definition of pasture such as \cite[Definition~6.19]{BL2}.
\end{remark}

\begin{remark} \label{rem:fissure-rule}
    If $B$ is a pasture, then we can break apart longer relations into three-term relations inductively. This procedure is known as \emph{fissure}. If $a_i \in B^\bullet$ and $a_0 \leqslant a_1 + \dots + a_n$ then there exists a $t \in B^\bullet$ for which $a_0 \leqslant a_1 + t$ and $t \leqslant a_2 + \dots + a_n$.
    A consequence of fissure is that $0 \leqslant a + b + c$ if and only if $-a \leqslant b + c$. A consequence of that consequence is that we can recover a pasture from its idyllic part.

    Because of this, we can also view pastures as a subcategory of idylls. Moreover, the relation $\preccurlyeq$ is the same as $\leqslant$ for pastures.
\end{remark}

For pastures, the two definitions of ``$a$ is a root of $f$'' are equivalent. The proof of this is a translation of Lemma~A in \cite{BL} to the language of pastures.

\begin{proposition} \label{prop:pasture-equivalence}
    If $B$ is a pasture and $f \in B[x]$ is a polynomial, then for any $a \in B^\bullet$, $0 \leqslant f(a)$ if and only if there exists a polynomial $g \in B[x]$ for which $f(x) \leqslant (x - a)g(x)$.
\end{proposition}
\begin{proof}
    First, if $a = 0$ then $f(0) = a_0$ and we have $f(0) = a_0 \geqslant 0$ if and only if each term in $f(x)$ is a multiple of $x$ and we can factor $f(x) = xg(x)$.

    Second, if $a \neq 0$, then by Remark~\ref{rem:fissure-rule}, $f(a) \geqslant 0$ means that there exists a sequence $t_1, t_2, \dots, t_n$ where $t_n = a^n$ and
    \begin{equation} \label{eq:factor-coef-equations}
        0 \leqslant b_0 + t_1 \text{ and } t_i \leqslant b_ia^i + t_{i + 1}, \text{ for } i = 1,\dots,n-1.
    \end{equation}
    In particular, we have the following sequence of inequalities:
    \begin{equation} \label{eq:ev-is-zero}
        0 \leqslant b_0 + t_1 \leqslant b_0 + b_1a + t_2 \leqslant \cdots \leqslant b_0 + b_1a + \dots + b_{n-1}a^{n-1} + a^n.
    \end{equation}

    Now, let us define a sequence $c_0, \dots, c_{n-1}$ by the equations $-a^ic_i = t_{i + 1}$ for $i = 0, \dots, n - 1$. Then the inequalities in \eqref{eq:factor-coef-equations} say
    \[
        0 \leqslant b_0 - c_0, \text{ and } -a^{i-1}c_{i-1} \leqslant b_ia^i - a^ic_i \iff b_i \leqslant c_i - ac_{i - 1}.
    \]
    These are exactly the inequalities which say that $f(x) \leqslant (x - a)g(x)$ where $g(x) = \sum c_ix^i$.

    Conversely, if we know that $f(x) \leqslant (x - a)g(x)$, then we can go backwards and construct a sequence $t_i$ such that the chain of inequalities in \eqref{eq:ev-is-zero} hold.
\end{proof}

\subsection{Multiplicities}
Let us return back to idylls and recall the definition of multiplicities.

\defmult*

Examples of factorizations are given in Appendix~\ref{sec:factorization}.

\subsubsection{Morphisms and multiplicities}
The next task is to show that morphisms preserve factorizations and hence multiplicities cannot decrease after applying a homomorphism. Additionally, we will verify that under isomorphism, multiplicities are the same and we will apply this to define multiplicities for initial forms.

\begin{proposition} \label{prop:mult-morph}
    Let $\vf : B \to B'$ be a homomorphism between two idylls. Let $f = \sum b_i x^i \in B[x]$ be a polynomial, let $a \in B^\bullet$ and let $a' = \vf(a), b_i' = \vf(b_i)$. Then
    \[
        \mult^B_a(f) \le \mult^{B'}_{a'}(\vf(f))
    \]
    where $\vf(f) = \sum b_i'x^i \in B'[x]$.
\end{proposition}

\begin{lemma}
    A homomorphism $\vf : B \to B'$ induces a homomorphism $\vf : B[x] \to B'[x]$ which is multiplicative. I.e.\ if $f \preccurlyeq gh$ then $\vf(f) \preccurlyeq \vf(g)\vf(h)$.
\end{lemma}

\begin{proof}
    Let us use the notation $a'$ for $\vf(a)$. It is a simple exercise to verify that $(ax^n)' \coloneqq a'x^n$ is a homomorphism between the two polynomial extensions $B[x]$ and $B'[x]$.

    To see that this homomorphism is multiplicative, first break apart the relation on $B[x]$ into a collection of relations on $B$ as follows.
    \[
        \sum a_kx^k \preccurlyeq \left( \sum b_i x^i \right)\left( \sum c_jx^j \right) \iff a_k \preccurlyeq \sum_{i + j = k} b_ic_j \text{ for all } k.
    \]
    Now apply $\vf$ everywhere to obtain
    \[
        a_k' \preccurlyeq \sum_{i + j = k} b_i'c_j' \text{ for all } k \implies \sum a_k'x^k \preccurlyeq \left( \sum b_i' x^i \right)\left( \sum c_j'x^j \right). \qedhere
    \]
\end{proof}

This result was first stated for hyperfields in \cite[Lemma~3.1]{G}. Proposition~\ref{prop:mult-morph} follows by applying this lemma to a sequence of factorizations of $f$ of maximal length.

Next, we look at how monomial transformations interact with multiplicities.

\begin{lemma} \label{lem:monomial-trans-mult}
    Let $\vf : B[x] \to B[x]$ be a monomial transformation given by $x \mapsto cx$. Then for any polynomial $f \in B[x]$ and $a \in B$,
    \[
        \mult^B_a (\vf(f)) = \mult^B_{ac} (f).
    \]
\end{lemma}

\begin{proof}[Proof]
    The proof of this lemma is similar to the proof of Proposition~\ref{prop:mult-morph}. First, we see that a factorization $f(x) \preccurlyeq (x - ca) g(x)$ yields a factorization
    \[
        f(cx) \preccurlyeq (cx - ca) g(cx) = (x - a) [cg(cx)].
    \]
    Then, we apply induction to obtain $\mult^B_{ac} (f) \leq \mult^B_{c} (\vf(f))$. The opposite inequality follows by considering the inverse transformation $x \mapsto c^{-1}x$.
\end{proof}

\begin{definition}\label{def:nonsplit-in-mult}
    Let $C \in \Ext^1(\Gamma, B)$, let $f \in C[x]$ be a polynomial, and let $a \in C^\bullet$ be a root of $f$ with valuation $\gamma_1$ and corresponding to the line $\ell = \{ \gamma_0 - i\gamma_1 : i \in \N\}$.

    We have $\In_a f \in B^{\gamma_0}[x]$ and by Lemma~\ref{lem:monomial-trans-mult}, the monomial substitution $x \mapsto ax$ in the definition of $\In_a f$ (\ref{def:nonsplit-in}) does not affect the multiplicity. Additionally, for any $c \in B^{\gamma_0}$, multiplication by $c^{-1}$ gives an isomorphism $B^{\gamma_0} \to B^0 = B^\times$ which again preserves multiplicity. Therefore, the quantity
    \[
        \mult^B_{\lc^\bullet(a)}(\In_a f) \coloneqq \mult^{C}_1(c^{-1} \In_a f)
    \]
    is well-defined. We take this as the general definition of a multiplicity for an initial form.
\end{definition}

For split extensions, this multiplicity agrees with the multiplicity of the initial form defined in \ref{def:split-lc-in}. This extends Remark~\ref{rem:initial-forms-identity}.

\begin{proposition} \label{prop:equivalent}
    If $C = B[\Gamma]$ is a split extension, and $a \in C^\bullet$ has valuation $\gamma$, then $\mult^B_{\lc^\bullet(a)}(\In_a f)$ as defined in \ref{def:nonsplit-in-mult} is equal to $\mult^B_{\lc^\bullet(a)}(\In_\gamma f)$ as defined in \ref{def:mult}.
\end{proposition}

\begin{proof}
    From Remark~\ref{rem:initial-forms-identity}, if $a = b^\gamma$, then
    \[
        \In_a f(x) = \In_{1^\gamma} f(bx) = 1^{\gamma_0} \In_{\gamma} f(bx).
    \]
    Next, from Definition~\ref{def:nonsplit-in-mult}, we defined
    \[
        \mult^B_{\lc^\bullet(a)} (\In_a f) = \mult^{B[\Gamma]}_1 \left(1^{-\gamma} \In_{1^\gamma} f(bx) \right) = \mult^{B[\Gamma]}_1 \left(\In_{\gamma} f(bx) \right).
    \]
    We want to check that computing this multiplicity in $B[\Gamma][x]$ rather than in $B[x]$ makes no difference.

    First, since $B$ embeds in $B[\Gamma]$, we have an inequality
    \[
        \mult^{B}_1 \left(\In_{\gamma} f(bx) \right) \le \mult^{B[\Gamma]}_1 \left(\In_{\gamma} f(bx) \right)
    \]
    by Proposition~\ref{prop:mult-morph}.

    Second, suppose we have some factorization $\In_{\gamma} f(bx) \leqslant (x - 1)g(x)$ in $B[\Gamma][x]$. And now remember that by definition of $N_{B[\Gamma]}$ (\ref{def:split-extension}), a relation holds if and only if it holds among just the terms of smallest valuation. I.e.\  if we let $\tilde g(x)$ be obtained from $g(x)$ by throwing out any higher order terms, then we have the relation $\In_{\gamma} f(bx) \preccurlyeq (x - 1)\tilde g(x)$ in $B[x]$. Therefore, by induction, we have
    \[
        \mult^{B}_1 \left(\In_{\gamma} f(bx) \right) = \mult^{B[\Gamma]}_1 \left(\In_{\gamma} f(bx) \right).
    \]

    We finish by observing that
    \[
        \mult^{B}_1 (\In_{\gamma} f(bx) ) = \mult^{B}_{b} (\In_{\gamma} f(x) ) = \mult^{B}_{\lc^\bullet(a)} (\In_{\gamma} f ). \qedhere
    \]
\end{proof}

\section{Hyperfields} \label{sec:hyperfields}
We defined hyperfields as idylls in Definition~\ref{def:hyperfield-idyll}. Or more specifically, we defined idylls of \emph{quotient} hyperfields. In this section, we make use of the language of pastures from the previous section to describe hyperfields in more detail. Then we will explain how our definition of tropical extension generalizes the semidirect product of Bowler and Su \cite{BS}.

\begin{definition} \label{def:hyperfield}
    A hyperfield is a pasture $H$, such that the \emph{hypersum} $a \boxplus b \coloneqq \{c : c \leqslant a + b\}$ is always nonempty and the operation $\boxplus$ is associative:
    \[
        (a \boxplus b) \boxplus c = \bigcup_{t \in a \boxplus b} t \boxplus c = \bigcup_{t \in b \boxplus c} a \boxplus t = a \boxplus (b \boxplus c).
    \]
    Here we are using the monadic laws discussed in (Remark~\ref{rem:hyperfield-not-nec-quots}).
\end{definition}

\begin{example}
    The tropical hyperfield is the hyperfield on $\R \cup \{\infty\}$ where $a \in b \boxplus c$ if the minimum of $a, b, c$ occurs at least twice. The tropical idyll is the idyllic part of this pasture.
\end{example}

\begin{example}
    The sign hyperfield is the hyperfield on $\S^\bullet = \{0, 1, -1\}$ and where addition is defined by
    \[
        \begin{array}{r|rcc}
            \boxplus & 0 & 1 & -1 \\\hline
            0 & 0 & 1 & -1 \\
            1 & 1 & 1 & \S^\bullet \\
            -1 & -1 & \S^\bullet & -1
        \end{array}
    \]
    The sign idyll is the idyllic part of this pasture.
\end{example}

\begin{definition}
    A hypergroup $H$ is a set $H$ together with a distinguished element $0$ and hypersum operation $\boxplus$ from $H \times H$ to the powerset of $H$ such that for all $x, y, z \in H$:
    \begin{itemize}
        \item $\boxplus$ is commutative and associative,
        \item $0 \boxplus x = \{x\}$,
        \item there exists a unique element $-x$ such that $0 \in x \boxplus (-x)$,
        \item $x \in y \boxplus z$ if and only if $-y \in (-x) \boxplus z$.
    \end{itemize}
\end{definition}

\begin{remark}
    Another definition (the standard one) of a hyperfield is that it is a hypergroup with a multiplication which distributes over hypersums and which has multiplicative inverses. I.e.\ hyperfields are monoids in the category of hypergroups.
\end{remark}

We now describe Bowler and Su's semidirect in a slightly-modified language. Because we work in a commutative setting, we can simplify some conditions required by non-commutativity.

\begin{definition}
    Let $B$ be a hyperfield, let $H = (H, 1, \cdot)$ be an Abelian group written multiplicatively, and let $\Gamma$ be an OAG. Suppose we have an exact sequence
    \begin{equation} \label{eq:BS-SES}
        1 \to B^\times \xrightarrow{\iota} H \xrightarrow{v} \Gamma \to 1,
    \end{equation}
    and we will assume that $\iota$ is the identity.

    Define $H^\bullet = H \cup \{0\}$ to be the monoid obtained by formally adding an absorbing element $0$ to $H$. Next, for each $\gamma \in \Gamma$, let $B^\gamma = v^{-1}(\gamma)$ as in Definition~\ref{def:torsors}.
 
    If $x, y \in B^\gamma \cup \{0\}$ and $c \in B^\gamma$, we can define $x \boxplus_\gamma y = \{z \in B^\gamma \cup \{0\} : (c^{-1}z) \in (c^{-1}x) \boxplus (c^{-1}y) \text{ in } B\}$. This hypersum is independent of $c$ because if $c_1, c_2 \in B^\gamma$ then multiplication by $c_1c_2^{-1}$ is an automorphism of $B$. This defines a hypersum on $B_\gamma \coloneqq B^\gamma \cup \{0\}$ and makes $B_\gamma$ into a hypergroup.

    The \emph{$\Gamma$-layering} $B \rtimes_{H, v} \Gamma$ of $B$ along this short exact sequence is a hyperfield whose underlying monoid is $H^\bullet$ and where $y \boxplus z$ is
    \begin{enumerate}[label=(H\arabic*)]
        \item \label{BS:y} $\{y\}$ if $v(y) < v(z)$
        \item \label{BS:z} $\{z\}$ if $v(z) < v(y)$
        \item \label{BS:equal} $y \boxplus_\gamma z$ if $v(y) = v(z) \eqqcolon \gamma$ and $0 \notin y \boxplus_\gamma z$
        \item \label{BS:x} $y \boxplus_\gamma z \cup \{x : v(x) > \gamma\}$ if $0 \in y \boxplus_\gamma z$
    \end{enumerate}
\end{definition}

\begin{proposition} \label{prop:BS-equiv}
    The Bowler-Su semidirect product $C \coloneqq B \rtimes_{H, v} \Gamma$ is a tropical extension in the sense of Definition~\ref{def:tropext}.
\end{proposition}

\begin{proof}
    Since the underlying monoid of $C$ is $H^\bullet$, the short exact sequence in equation \eqref{eq:BS-SES} is the same as
    \[
        1 \to B^\times \xrightarrow{\iota} C^\times \xrightarrow{v} \Gamma \to 1    
    \]
    in Definition~\ref{def:tropext}.

    Next, let us check that $\im(\iota) = \eq(v, 1)$, i.e.\ that $B$ is a full subblueprint of $C$. By construction, we have $B^\bullet = \eq(v, 1)^\bullet$ as monoids. We need to check relations. If $x, y, z \in B^\bullet$, then we can take $c = 1_C$ in the definition of $\boxplus_0$ to see that $x \in y \boxplus z$ in $B$ if and only if $x \in y \boxplus_0 z$ in $C$ if and only if $x \in y \boxplus z$ in $C$ (compare \ref{BS:equal}). Therefore $\im(\iota)$ is a full subblueprint and hence equal to $\eq(v, 1)$.

    Finally, we need to check that $x \in y \boxplus z$ if and only if this holds when looking at just the terms of minimal valuation.
    \begin{itemize}
        \item If $v(x) = v(y) < v(z)$ then $x \in y \boxplus z$ if and only if $x \in y \boxplus 0$ by \ref{BS:y}.
        \item If $v(x) = v(z) < v(y)$ then $x \in y \boxplus z$ if and only if $x \in 0 \boxplus z$ by \ref{BS:z}.
        \item If $v(y) = v(z) < v(x)$ then $x \in y \boxplus z$ if and only if $0 \in y \boxplus z$ by \ref{BS:x}.
        \item If the minimum valuation does not occur at least twice, then vacuously there are no $x, y, z$ such that $x \in y \boxplus z$ and neither do we have any the relations $0 \in y \boxplus 0, 0 \in 0 \boxplus z$ or $x \in 0 \boxplus 0$.
    \end{itemize}
    So we conclude that $C$ is a tropical extension.
\end{proof}

\section{Lifting Theorem for Multiplicities}
We have defined tropical extensions, initial forms and multiplicities and seen that our definitions agree with each other. Now we are ready to prove the main theorem, and we will do over the course of this section. First, let us recall the definition of \emph{wholeness} from the introduction.

\begin{definition}
    An idyll $B$ is \emph{whole} if for every pair of elements $a, b \in B^\bullet$, there exists at least one element $c$ such that $c \preccurlyeq a + b$.
\end{definition}

Recall that in language of hyperfields or partial fields, we have a notion of sum sets: $a \boxplus b = \{c : c \preccurlyeq a + b\}$ (Example~\ref{ex:sumset}). A pasture is whole if every sum is non-empty. So hyperfields and fields are always whole but partial fields are only whole if they are fields. Whole idylls are closely related therefore to hyperfields.

\begin{remark}
    If $B$ is whole, then any tropical extension by $B$ is also whole. We have two cases. First, if $v^\bullet(a) = v^\bullet(b)$, then both $a$ and $b$ live in some torsor $B^\gamma$. Now take $c \in B^\gamma$ and consider $c^{-1}a, c^{-1}b \in B^0 = B^\times$. Since $B$ is whole, we can find an element $c'$ such that $c' \preccurlyeq c^{-1}a + c^{-1}b$ and then multiply both sides by $c$ to get $cc' \preccurlyeq a + b$.

    Otherwise, if $v^\bullet(a) < v^\bullet(b)$, say, then $a \preccurlyeq a + b$ because this relation is true among the minimum valuation terms.
\end{remark}

This brings us to the main theorem. Let us recall.

\mainthm*

Theorem~\ref{thm:split}, which describes the split case, is a direct corollary of this theorem in light of Proposition~\ref{prop:equivalent}.

\begin{lemma} \label{lem:ev}
    Let $C \in \Ext^1(\Gamma, B)$ and define the idyllic subblueprint $\OC$ of $C$ to be the induced subblueprint corresponding to the submonoid $\{c \in C^\bullet : v^\bullet(c) \ge 0\}$. Let $\ev_0 : \OC \to B$ be the map which ``evaluates $t$ at $0$'' meaning
    \[
        \ev_0^\bullet(c) =
        \begin{cases}
            c & \text{if } c \in B^0, \\
            0 & \text{if } c \in B^\gamma, \gamma > 0.
        \end{cases}
    \]
    Then $\ev_0$ is a morphism.
\end{lemma}

The language of ``evaluating $t$ at $0$'' comes from the split case wherein $\ev_0^\bullet(bt^\gamma) = b0^\gamma$ with the usual convention that $0^0 = 1$.

\begin{proof}
    Simple case checking shows that $\ev_0 : \OC^\bullet \to B^\bullet$ is a morphism. It is left then to check that $\ev_0(N_{\OC}) \subseteq N_B$.

    Given $\sum c_i \in N_{\OC}$, there are two cases. First, if every $c_i$ has a positive valuation, then $\ev_0(\sum c_i) = 0_B \in N_B$. Second, suppose that $I = \{i : v^\bullet(c_i) = 0\}$ is non-empty. Then by definition of $N_C$, we must have $\sum_I c_i \in N_{\OC} \subset N_C$ since $0$ is the minimum valuation. But $\sum_I c_i$ also lives in $B^0 = B^\times$ so we get $\sum_I c_i = \ev_0(\sum c_i) \in N_B$.
\end{proof}

\begin{lemma} \label{lem:main-first-ineq}
    $\mult^C_a(f) \le \mult^{B}_{\lc^\bullet(a)} (\In_a(f))$.
\end{lemma}
\begin{proof}
    Recall that the initial form of a polynomial $f = \sum c_ix^i \in C[x]$ is defined as
    \[
        \In_a(f) = \sum_{I} \lc^\bullet(c_i)(ax)^i \in B^{\gamma_0}[x]
    \]
    where $I = \{i : v^\bullet(c_ia^i) \text{ is minimal}\}$ and $\gamma_0$ is that minimum value. In other words, this initial form is obtained from the polynomial $g(x) = f(ax)$ by restricting the sum to $I$. Observe that by Lemma~\ref{lem:monomial-trans-mult}, we have $\mult^C_1 g = \mult^C_a f$.

    Next, choose any $c \in B^{\gamma_0}$. By Proposition~\ref{prop:mult-morph}, and the fact that multiplication by $c$ is invertible, we have $\mult^C_1 c^{-1} g = \mult^C_1 g$, independent of the choice of $c$.

    Now observe that $c^{-1} g \in \OC[x]$ and $\ev_0(c^{-1} g) = c^{-1} \In_a(f)$. So because $\ev_0$ is a morphism (Lemma~\ref{lem:ev}), we must have
    \[
        \mult^C_1 c^{-1} g  \le \mult_1^B c^{-1} \In_a(f).
    \]
    By what we have said, the left side of this inequality is $\mult^C_a(f)$ and the right side is $\mult^{B}_{\lc^\bullet(a)} (\In_a(f))$.
\end{proof}

\begin{lemma} \label{lem:assume-gamma-is-R}
    In proving Theorem~\ref{thm:main}, we may assume that $\Gamma = \R$.
\end{lemma}
\begin{proof}
    We would like to appeal to Lemma~\ref{lem:inductive-in} and induction, but in order to do so, we need a finite-rank hypothesis. We can get this by considering the subgroup generated by the coefficients.

    Specifically, let $f \preccurlyeq (x - a)g_0$ and $g_k \preccurlyeq (x - a)g_{k + 1}$ be a sequence of factorizations of maximal length. Let $\Gamma'$ be the subgroup generated by the coefficients of $f, g_0, g_1, \dots$. If we define $C' = \bigcup_{\gamma \in \Gamma'} B^\gamma$, then $\mult^C_a f = \mult^{C'}_a f$ and $\operatorname{rank} \Gamma' < \infty$.
\end{proof}

We now have everything in hand to prove the lifting theorem.

\liftingthm*


\begin{proof}
    First, by making monomial substitutions $x \mapsto ax$ or $x \mapsto a^{-1}x$ in the appropriate places, we are going to assume that $a = 1$. Also, by multiplying by $c$ or $c^{-1}$ for some $c \in B^{\gamma_0}$, we are going to assume that the minimal valuation of the terms in $f$ or $\In_1 f$ is exactly $0$. As a consequence, we now have the identity $\In_1 f = \ev_0 f \in B[x]$ and by the last half of the proof of Proposition~\ref{prop:equivalent} regarding factorization in $B$ versus in $B[\Gamma]$, there is no loss of generality treating this as a polynomial over $B$ rather than over $C$.

    From Lemma~\ref{lem:assume-gamma-is-R}, we can assume that $\Gamma = \R$ and this will allow us to consider the Newton polygon as defined in Section~\ref{sec:newt-poly}. We will break up the polynomial into three sections. Let $i_0 = \min\{i : v^\bullet(c_i) = 0\}$ and $i_1 = \max\{i : v^\bullet(c_i) = 0\}$. Let $I_L = \{i : i < i_0\}$ be the \emph{left} interval, let $I_M = \{i : i_0 \le i \le i_1\}$ be the \emph{middle} interval, let $I_R = \{i : i > i_1\}$ be the \emph{right} interval, and as always, we define $I = \{i : v^\bullet(c_i) = 0\}$. See Figure~\ref{fig:construction} for a visual.

    So suppose we have $\In_1 f \preccurlyeq (x - 1) g$ where $f = \sum c_ix^i$ and $g = \sum d_i x^i$. In what follows, we will treat the coefficients as infinite sequences by defining the terms not appearing in the sum to be $0$. Then, we will modify the coefficients of $g$ by redefining them in such a way that if $\tilde{g}$ is obtained from $g$ by redefining some $d_i$'s then $\In_1 \tilde{g} = \ev_0 \tilde{g} = g$. In particular, if $d_i$ is non-zero then we do not touch it and if $d_i = 0$ then it might be redefined to another element of positive valuation.

    \begin{claim} \label{claim:support}
        The support of $g$ is contained in $i_0, \dots, i_1 - 1$ and $d_{i_0} \neq 0$ and $d_{i_1 - 1} \neq 0$.
    \end{claim}

    These facts are the same as for polynomials over a field. For instance, we know that $\deg g = \deg(\In_1 f) - 1$ because there are no zero-divisors in an idyll. For the smallest non-zero coefficient, we can write $f = x^{d_0}f_0$ and $g = x^kg_0$ where $k$ is maximal. If $k \neq d_0$ then we can divide both sides of $\In_1 f \preccurlyeq (x - 1) g$ by $x^{\min\{k, d_0\}}$ and set $x = 0$ (i.e.\ consider the relation in degree $0$) to get a contradiction.
    \claimqed

    Next, let us describe how to lift $g$ on each of the left, middle and right parts of the Newton polygon. We will start with the middle since $\In_1 f$ is supported there.

    \begin{claim}
        We have $\sum_{I_M} c_i x^i \preccurlyeq (x - 1) g$ (with no changes to $g$).
    \end{claim}

    To say that $\In_1 f \preccurlyeq (x - 1) g$ means that $c_i \preccurlyeq d_{i - 1} - d_i$ for $i \in I$. But it also means that $0 \leqslant d_{i - 1} - d_i$ for $i \in I_M \setminus I$. Now, if $c \in C^\bullet$ has a positive valuation, then we also have $c \preccurlyeq d_{i - 1} - d_i$ since by definition of $N_C$, a relation holds in $C$ if and only if it holds among just the terms of minimal valuation. Therefore $c_i \preccurlyeq d_{i - 1} - d_i$ for all $i \in I_M$.
    \claimqed

    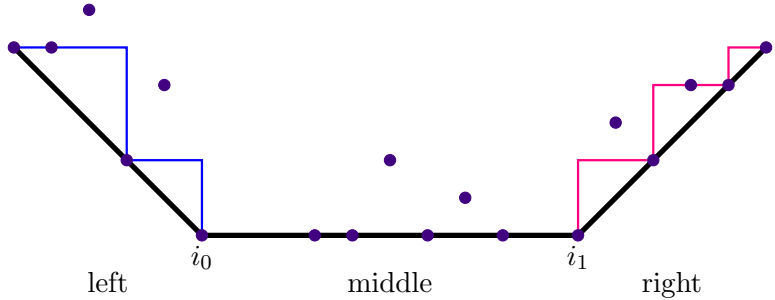
\begin{figure}[tbp]
        \centering

        \begin{tikzpicture}[scale=0.5]
            \draw[line width=0.3mm, blue] (0, 5) -- (3, 5) -- (3, 2) -- (5, 2) -- (5,0);

            \draw[line width=0.3mm, mypink] (15, 0) -- (15, 2) -- (17, 2) -- (17, 4) -- (19, 4) -- (19, 5) -- (20, 5);

            \draw[line width=0.7mm] (0, 5) -- (5, 0) -- (15, 0) -- (20, 5);

            \draw (2.5, 0) node[below=10] {left};
            \draw (10, 0) node[below=10] {middle};
            \draw (17.5, 0) node[below=10] {right};

            \draw (5, 0) node[below] {$i_0$};
            \draw (15, 0) node[below] {$i_1$};

            \foreach \x/\y in {0/5, 1/5, 2/6, 3/2, 4/4, 5/0,
            8/0, 9/0, 10/2, 11/0, 12/1, 13/0, 15/0, 20/5,
            16/3, 17/2, 18/4, 19/4}
                \filldraw[mypurple] (\x,\y) circle (1.5mm);

        \end{tikzpicture}

        \caption{Newton polygon describing the construction}
        \label{fig:construction}
    \end{figure}

    Next we look at $I_L$. Here, we will begin by redefining $d_0 = -c_0$. After that, there are three kinds of points in the Newton polygon/indices $i$. First, there are the points where $v^\bullet(c_i) < \min\{v^\bullet(c_k) : k < i\}$ (in the diagram these are where the \textcolor{blue}{blue staircase} on the left descends). Second, there are points where $v^\bullet(c_i) = \min\{v^\bullet(c_k) : k < i\}$ (points along the flats of the staircase). Then finally, there are points where $v^\bullet(c_i) > \min\{v^\bullet(c_k) : k < i\}$ (points above the staircase).

    For the points where the staircase descends, we define $d_i = -c_i$. For the points on or above the flats of the staircase, inductively define $d_i$ to be any element for which $d_i \preccurlyeq d_{i - 1} - c_i$ (making use of the wholeness axiom). Note that because $v$ is a valuation, $v^\bullet(d_i) \ge \min\{v^\bullet(d_{i - 1}), v^\bullet(c_i)\}$ and so these have a positive valuation if $i < i_0$. We make these redefinitions for all $i \in I_L$.

    \begin{claim} \label{claim:left}
        We have $\In_1 \tilde{g} = g$ and $\sum_{I_L \cup I_M} c_i x^i \preccurlyeq (x - 1)\tilde{g}$.
    \end{claim}

    For the first part of the claim, we note that based on how we have redefined $d_i$, any time we changed a value, it was a zero value becoming a value with a positive valuation.

    We need to check that $c_i \preccurlyeq d_{i - 1} - d_i$ for $i = 1, \dots, i_0 - 1$. We have already verified this for $i = i_0, \dots, i_1$ with the exception that now for $i_0$, we have $d_{i_0 - 1} \neq 0$, this change is handled below in Case 1.
    
    To start, the relation $c_0 \preccurlyeq 0 - d_0$ holds by definition. From there, we proceed by induction.

    Case 1: if we are at a point where the staircase descends, then we have $v^\bullet(c_i) = v^\bullet(d_i) < v^\bullet(d_{i - 1})$. Here the relation $c_i \preccurlyeq d_{i - 1} - d_i$ holds because it holds among the minimal valuation terms: $c_i \preccurlyeq -d_i$.

    Case 2: if we are on or above one of the flats, then the definition $d_i \preccurlyeq d_{i - 1} - c_i$ is equivalent to $c_i \preccurlyeq d_{i - 1} - d_i$.
    \claimqed

    Lastly, we need to define $d_i$ for $i \in I_R$ and also $d_{i_1}$.
    We will start by defining another staircase function: $j(i) = \min\{k : v^\bullet(c_k) \text{ is minimal and } k > i \}$. In the diagram, $j(i)$ is the next $x$-coordinate along the \textcolor{mypink}{pink staircase} on the right. When $j(i - 1) \neq j(i)$, we will define $d_{i - 1} = c_i$. Otherwise, we let $d_i$ be any element satisfying $d_i \preccurlyeq d_{i - 1} + c_{i+1}$.

    \begin{claim} \label{claim:right}
        We have $\In_1 \tilde{g} = g$ and $\sum c_i x^i \preccurlyeq (x - 1)\tilde{g}$.
    \end{claim}

    As with the last claim, the first part just comes down to verifying that any time we have redefined a zero-valued $d_i$, that the new value has a positive valuation. This is true here because $v^\bullet(c_i) > 0$ for any $i > i_1$ by definition of $i_1$.

    Now we need to check that $c_i \preccurlyeq d_{i - 1} - d_i$ for $i = i_1 - 1, i_1, i_1 + 1, \dots$. The indices in $I_L \cup I_M$ have already been checked except for $i_1$ since we have given a new value to $d_{i_1}$. Again we have two cases:

    Case 1: if $j(i - 1) \neq j(i)$ then that is because $v^\bullet(c_i) < v^\bullet(c_k)$ for any $k > i$. Here we have $d_{i - 1} = c_i$ and $v^\bullet(d_i) > v^\bullet(d_{i - 1})$. Thus $c_i \preccurlyeq d_{i - 1} - d_i$ because the minimal valuation part of this relation is $c_i \preccurlyeq d_{i - 1}$.

    Case 2: if $j(i - 1) = j(i)$ then we proceed by induction. The sequence $j(i)$ is non-decreasing and as a base case, we know that $c_i \preccurlyeq d_{i - 1} - d_i$ every time $j(i - 1) < j(i)$. Given $c_i \preccurlyeq d_{i - 1} - d_i$ and $j(i - 1) = j(i)$, we will check that $c_{i + 1} \preccurlyeq d_{i} - d_{i + 1}$. Indeed, this is exactly how we defined $d_{i + 1}$, so that this relation would hold.
\end{proof}

Finally, we finish the proof of Theorem~\ref{thm:main}. We do what we have done before: take a factorization sequence of $\In_1 f$ of maximal length and lift it to a factorization sequence of $f$. That gives us
\[
    \mult^C_a(f) \ge \mult^B_{\lc^\bullet(a)} (\In_a f).
\]
Combining this with Lemma~\ref{lem:main-first-ineq}, we obtain Theorem~\ref{thm:main}.

\section{Examples and Connections} \label{sec:connections}
Theorems~\ref{thm:split} and \ref{thm:main} imply some results of previous papers. First of all, it gives a new proof of Theorem~D from Baker and Lorscheid's paper \cite{BL}.

\begin{corollary}
    Let $f \in \T[x]$ and for $a \in \R$, define $v_a(f)$ to be $j - i$ if the edge in the Newton polygon of $f$ with slope $-a$ has endpoints $(i, c_i)$ and $(j, c_j)$. If there is no such edge, define $v_a(f) = 0$.
    
    Given this, we have $\mult^\T_a(f) = v_a(f)$.
\end{corollary}
\begin{proof}
    By Theorem~\ref{thm:split}, we have $\mult^\T_a(f) = \mult^\K_1(\In_a f)$ and $\In_a f$ is the sum of $x^k$ over all $k$ such that $(k, c_k)$ is in the edge of slope $-a$. And for the Krasner idyll, we have $\mult^\K_1(x^i + \dots + x^j) = j - i$ (Example~\ref{ex:krasner-factorization}).
\end{proof}

Next, let us have a look at the extension $\TR = \S[\R] \in \Ext^1(\R, \S)$ which was the main focus of \cite{G}.

\begin{corollary}
    Let $f \in \TR[x]$ and $a = (+1)^\gamma \in \TR^\bullet$. Then $\mult^\TR_a f$ equals the number of sign changes among the coefficients corresponding to points in $\newt f$ inside the edge of slope $-\gamma$.
\end{corollary}
\begin{proof}
    By Theorem~\ref{thm:split}, we have $\mult^\TR_a f = \mult^\S_{+1} \In_\gamma f$ where with the notation we have been using, $\In_\gamma f = \sum_I \lc^\bullet(c_i)x^i$ and $I$ is the set of all $i$ such that $(i, v^\bullet(c_i))$ is contained in the edge of $\newt f$ with slope $-\gamma$.

    Next, by \cite[Theorem~C]{BL}, $\mult^\S_{+1} \In_\gamma f$ is equal to the number of sign changes in the sequence $(\lc^\bullet(c_i) : i \in I)$---ignoring zeroes.
\end{proof}

\begin{remark} \label{rem:phase-example}
    The next place to look would be at multiplicities over $\TC$. We still have that $\mult^\TC_a f = \mult^\P_{\lc^\bullet(a)} \In_{v^\bullet(a)} f$ but there is no existing simple description of $\mult^\P$. In fact, polynomials over $\P$ have some pathologies as pointed out by Philipp Jell \cite[Remark~1.10]{BL}: the polynomial $x^2 + x + 1 \in \P[x]$ has a root at $e^{i\theta}$ for all $\pi/2 < \theta < 3\pi/2$. In contrast, polynomials over $\K$ or $\S$ or tropical extensions thereby, can only have finitely many roots.
\end{remark}

\subsection{Higher rank}
Combining Lemma~\ref{lem:inductive-in} with the main theorem, tells us how to compute multiplicities of polynomials in the context of a higher-rank valuation.

\begin{example} \label{ex:higher-rank}
    Consider the following polynomial over $\C(s,t)$ with valuation $v^\bullet(s^mt^n) = (m, n) \in (\R^2, \lex)$:
    \begin{align*}
        f&= (x - t)(x - s)(x - st)(x - 2st) \\
        &= x^4 \\
        &\quad - (t + s + 3st)x^3 \\
        &\quad + (st + 3st^2 + 3s^2t + 2s^2t^2)x^2 \\
        &\quad - (3s^2t^2 + 2s^2t^3 + 2s^3t^2)x \\
        &\quad + 2s^3t^3.
    \end{align*}

    Suppose we want to know how many roots of $f$ have valuation $(1,1)$. I.e.\ what is $\mult^{\T_2}_{(1,1)} \operatorname{trop}(f)$ where
    \[
        \operatorname{trop}(f) = (3,{\color{blue}3}) + (2,{\color{blue}2})x + (1,{\color{blue}1})x^2 + (0,{\color{blue}1})x^3 + x^4 \in \T_2[x]?
    \]
    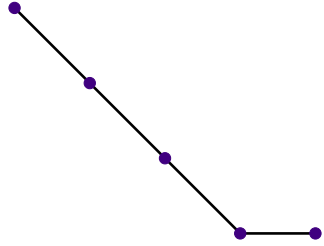
\begin{figure}[htbp]
        \centering

        \begin{tikzpicture}
            \draw[line width=0.35mm] (0, 3) -- (3, 0) -- (4, 0);

            \foreach \x/\y in {0/3, 1/2, 2/1, 3/0, 4/0}
                \filldraw[mypurple] (\x,\y) circle (0.75mm);
        \end{tikzpicture}

        \caption{Newton polygon of $f$ with respect to $v_s$ in Example~\ref{ex:higher-rank}}
        \label{fig:s-val}
    \end{figure}
    By Lemma~\ref{lem:inductive-in}, we start by considering the $s$-valuation $v_s(s^mt^n) = m$ and draw a Newton polygon based on the first coordinate of each coefficient (Figure~\ref{fig:s-val}). Then, we pick out the line segment of slope $-1$ to create the initial form
    \[
        \In_1 \operatorname{trop}(f) = {\color{blue}3} + {\color{blue}2}x + {\color{blue}1}x^2 + {\color{blue}1}x^3 \in \T[x].
    \]

    Next, we draw the Newton polygon of this initial form with respect to the $t$-valuation (Figure~\ref{fig:t-val}).
    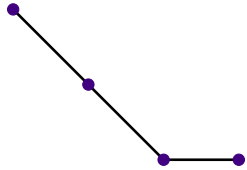
\begin{figure}[htbp]
        \centering

        \begin{tikzpicture}
            \draw[line width=0.35mm] (0, 3) -- (2, 1) -- (3, 1);

            \foreach \x/\y in {0/3, 1/2, 2/1, 3/1}
                \filldraw[mypurple] (\x,\y) circle (0.75mm);
        \end{tikzpicture}

        \caption{Newton polygon of $\In_1 \operatorname{trop}(f)$ in Example~\ref{ex:higher-rank}}
        \label{fig:t-val}
    \end{figure}
    Here we can take another initial form to get $\In_1(\In_1 \operatorname{trop}(f)) = 1 + x + x^2 \in \K[x]$. So $\mult^{\T_2}_{(1,1)} \operatorname{trop}(f) = 2$. \qedhere
\end{example}

\subsection{Connection to polynomials over fields}
Let us summarize what is know about a question which has been discussed before in \cite{BL} and \cite{G}: given a morphism $\vf$ from a field $K$ to an idyll $B$, and a polynomial $F \in K[x]$ lying over $f \in B[x]$, what can we say about multiplicities in $K$ compared to in $B$?

There are two questions here: local and global. Locally, we have the following inequality \cite[Proposition~B]{BL}:
\begin{equation} \label{eq:prop-b-ineq}
    \mult^B_b f \ge \sum_{a \in \vf^{-1}(b)} F.
\end{equation}

Globally, we know that the sum of the multiplicities in $B$ might be infinite (Remark~\ref{rem:phase-example}).

Here we will give a partial answer to the question that Baker and Lorscheid asked about when a hyperfield satisfies the degree bound, which in Definition~\ref{def:degreebound} we defined as:
\[
    \sum_{b \in B} \mult^B_b f \le \deg f
\]
for all polynomials $f \in B[x]$. By their Proposition~B, the corollary of this bound is that \eqref{eq:prop-b-ineq} becomes an equality.

\degreeboundthm*

\begin{proof}
    Let $f \in C[x]$ be a polynomial. Since any one polynomial only requires a finite-rank value group to define, we are going to again assume that $\Gamma = \R$, use induction to extend to any finite-rank value group, and then use the fact that any polynomial lives in a finite-rank sub-extension.

    With this reduction, consider the polynomial $v(f) \in \T[x]$ via the morphism $v : C \to \T$. The Newton polygon will have a finite number of edges and hence a finite number of non-monomial initial forms, say $\In_{\gamma_k} v(f)$ for $k = 1,\dots,d$. Now, let $a_k \in B^{\gamma_k} \subseteq C$ be a representative of $\gamma_k$. Each initial form is not-necessarily-canonically isomorphic to a polynomial in $B[x]$ and we are going to use the degree bound in $B$ to get a bound in $C$.

    First of all, we partition the roots of $f$ by valuation so
    \[
        \sum_{a \in C} \mult^C_a f = \mult^C_0 f + \sum_{k = 1}^d \sum_{a \in B^{\gamma_k}} \mult^C_a f.
    \]
    Next, we will show that
    \begin{equation} \label{eq:degree-bound-in}
        \sum_{a \in B^{\gamma_k}} \mult^C_a f \le \deg \In_{a_k} f - \mult^C_0 \In_{a_k} f.
    \end{equation}
    This will suffice to prove the theorem because
    \[
        \sum_{k = 1}^d \left( \deg \In_{a_k} f - \mult^C_0 \In_{a_k} f \right) = \deg f + \mult^C_0 f.
    \]
    This holds because the width of the Newton polygon is equal to the sum of the width of each edge in the polygon.

    To show \eqref{eq:degree-bound-in}, apply some transformation $c_k^{-1}f(a_kx)$ to get a polynomial in $\OC[x]$. Then apply $\ev_0$ to get a polynomial in $B[x]$. By Theorem~\ref{thm:main} and Theorem~\ref{thm:lifting}, there is a equality between multiplicities of valuation $0$ roots of $c_k^{-1}f(a_kx)$ and non-zero roots of $\ev_0(c_k^{-1}f(a_kx)) \in B[x]$. This shows \eqref{eq:degree-bound-in}.
\end{proof}

Since we know that the Krasner and sign hyperfields satisfy the degree bound, we have the following corollary.

\begin{corollary} \label{cor:degree-bound}
    If $B$ is either a field or $\K$ or $\S$, and $C \in \Ext^1(\Gamma, B)$ then $C$ satisfies the degree bound. For example, this applies to $C = \K, \S, \T, \TR$ and the higher-rank versions $\T_n = \K[\R^n]$ and $\S[\R^n]$.
\end{corollary}
\begin{proof}
    Combine Theorem~\ref{thm:degreebound} with \cite[Proposition~B]{BL}.
\end{proof}

Based on a classification of Bowler and Su, we can conclude that so-called \emph{stringent} hyperfields satisfy the degree bound.

\begin{definition}
    A hyperfield (idyll) is stringent if the sum-set $a \boxplus b = \{c : c \curlyeqprec a + b\}$ is a singleton if $a \neq b$.
\end{definition}

\stringentcor*

\begin{proof}
    Combine Corollary~\ref{cor:degree-bound} with Bowler and Su's classification of stringent hyperfields \cite[Theorem~4.10]{BS}.
\end{proof}

\subsubsection{Some open questions}
Corollary~\ref{cor:stringent} leads to several interesting questions.
\begin{question}
    Is there a more direct proof of Corollary~\ref{cor:stringent} which does not rely on Bowler and Su's classification?
\end{question}

\begin{question}
    Is the converse of Corollary~\ref{cor:stringent} true? I.e.\ if a hyperfield satisfies the degree bound, is it necessarily stringent?
\end{question}

\begin{question}
    Baker and Zhang show that a hyperfield is stringent if and only if its associated idyll satisfies a ``strong-fusion axiom'' \cite[Proposition~2.4]{BZ}. If the previous question has a positive answer, can we extend that to pastures or idylls with an additional axiom like strong-fusion?
\end{question}

\appendix
\section{Factorization Rules} \label{sec:factorization}
Within Baker and Lorscheid's paper \cite{BL}, the author's previous paper \cite{G}, and a paper of Agudelo and Lorscheid \cite{AL} are some descriptions of various division algorithms. Agudelo and Lorscheid spell out these algorithms explicitly and in the other two the algorithms are hidden inside the proofs. In this section, we describe these algorithms and explain where and how they appear in each of the aforementioned papers.

\begin{example} \label{ex:krasner-factorization}
    Over the Krasner hyperfield, with $m < n$, we have the following factorization:
    \[
         x^m + \text{any intermediate terms} + x^n \preccurlyeq (x + 1)(x^m + x^{m + 1} + x^{m + 2} + \dots + x^{n - 1}).
    \]
    Moreover, this factorization is optimal (the multiplicity of the quotient is exactly $1$ less). Therefore $\mult^\K_1 f = n - m$ for any polynomial with highest term $x^n$ and lowest term $x^m$.
\end{example}

The existence of this rule was alluded to in \cite{BL} but not spelled out. This rule is easy to verify and this verification is left to the reader.

\begin{example} \label{ex:negative-root-factorization}
    Let $f = \sum s_ix^i$ be a polynomial over the sign hyperfield with no intermediate zeroes between the lowest and highest term. Let $i_0$ be the smallest index for which $s_{i} = s_{i+1}$. Define a new sequence of signs by ``squishing together'' $s_{i_0}$ and $s_{i_0 + 1}$, so
    \[
        \tilde{s_i} = \begin{cases}
            s_i &\text{if } i \le i_0, \\
            s_{i + 1} &\text{if } i > i_0.
        \end{cases}
    \]
    Then $g = \sum \tilde{s_i}x^i$ is a quotient of $f$ by $(x + 1)$ and $\mult^\S_{-1} g$ is exactly one less than $\mult^\S_{-1} f$.

    For instance,
    \[
        1 - x + x^2 {\color{mypink} {}- x^3 - x^4} - x^5 + x^6 \preccurlyeq (1 + x)(1 - x + x^2 {\color{mypink} {}- x^3} - x^4 + x^5). \qedhere
    \]
\end{example}

\begin{example} \label{ex:positive-root-factorization}
    If we apply the previous rule to factoring out $(x - 1)$ by making the substitution $x \mapsto -x$ before and after, we get this rule:
    \[
        \tilde{s_i} = \begin{cases}
            -s_i &\text{if } i \le i_0, \\
            s_{i + 1} &\text{if } i > i_0
        \end{cases}
    \]
    where $i_0$ is now the smallest index where $s_i \neq s_{i + 1}$.
    From the previous example, if we substitute $x \mapsto -x$, the odd coefficients flip. Then, after ``squishing'' the parity changes so we get a different sign before and after the squish.

    For example,
    \[
        1 + x + {\color{mypink} x^2 - x^3} {}+ x^4 - x^5 \preccurlyeq (-1 + x)(-1 - x {\color{mypink} {}- x^2} + x^3 - x^4). \qedhere
    \]
\end{example}

Examples~\ref{ex:negative-root-factorization} and \ref{ex:positive-root-factorization} follow from a more general description which we will see next.

\begin{example}
    Let $f = \sum s_ix^i$ be a polynomial over the sign hyperfield, which for simplicity we assume has a nonzero constant term (otherwise factor out a monomial). Then as in Example~\ref{ex:positive-root-factorization}, let $i_0$ be the smallest index $i$ such that $s_{i} \neq s_{i_0}$.

    Then from left-to-right, define
    \[
        \tilde{s}_i = -s_i = -s_0 \text{ for } i \le i_0
    \]
    and for $i > i_0$, let $\tilde{s}_i = s_{j(i)}$ where $j(i) = \min\{j : j > i \text{ and } s_j \neq 0\}$ (i.e.\ the next non-zero coefficient after $s_i$). Then $g = \sum \tilde{s}_ix^i$ is a quotient of $f$ by $x - 1$ and $g$ has exactly one less sign change---equivalently one less positive root.
\end{example}

This rule first appeared in Baker and Lorscheid's paper \cite[proof of Theorem~C]{BL}. In the author's previous paper, this rule is extended to the tropical real hyperfield \cite[proof of Theorem~A]{G}. In Agudelo and Lorscheid's paper, the rule is adapted to apply to both factoring out $x - 1$ and $x + 1$.

In the context of this paper, the rule is obtained from the proof of Theorem~\ref{thm:lifting} in the paragraphs before Claims~\ref{claim:left} and \ref{claim:right} where we interpret $i_0, i_0 + 1$ as the ``middle.'' For instance, the function $j(i)$ defined in the previous example is a sibling of the function $j(i)$ defined before Claim~\ref{claim:right}. For the left portion, the rule we gave was $\tilde{s}_i \preccurlyeq \tilde{s}_{i - 1} - s_i$ which is certainly true if we define $\tilde{s}_i = -s_0 = -s_1 = -s_2 = \cdots = -s_{i_0}$ for all $i \le i_0$.

\begin{example}
    Let $f$ be a polynomial over the tropical hyperfield and let $a \in \T^\times$ be a root of $f$. Then to get a quotient of $f$ by $(x + a)$, first replace $f$ by $f(ax)$. Then apply Example~\ref{ex:krasner-factorization} to factor the initial form $\In_0 f$. Then lift that factorization to a factorization of $f$ by using the staircase rules illustrated in Figure~\ref{fig:construction} and described in the proof of Theorem~\ref{thm:lifting}.

    Specifically, on the left, let $d_i = \min\{d_{i - 1}, c_i\}$ and on the right, let $d_i = c_{j(i)}$ where $j(i) = \min\{j : j > i \text{ and } c_j \text{ is minimal}\}$.
\end{example}
This rule is also described in \cite{AL} without first making the substitution $f \mapsto f(ax)$. In \cite{BL}, an entirely different approach is given to tropical polynomials via looking at polynomial functions.

\emergencystretch=1em 
\printbibliography

\end{document}